\documentclass[12pt]{amsart}
\usepackage{amsmath,amsfonts,amsthm,amssymb,nicefrac}
\usepackage{graphicx,color,url}
\usepackage{xcolor}
\usepackage{accents}
\usepackage{enumerate}
\usepackage{comment}
\usepackage{hyperref,tikz}
\usepackage{wrapfig}
\usetikzlibrary{calc,angles,quotes}

\newtheorem{theorem}{Theorem}
\newtheorem{lemma}[theorem]{Lemma}

\newtheorem{proposition}[theorem]{Proposition}
\theoremstyle{definition}
\newtheorem{remark}[theorem]{Remark}

\newtheorem{definition}[theorem]{Definition}

\numberwithin{equation}{section}\numberwithin{theorem}{section}
\allowdisplaybreaks
\newcounter{stepctr}
{\end{list}}

\def\XXint#1#2#3{{\setbox0=\hbox{$#1{#2#3}{\int}$}
 \vcenter{\hbox{$#2#3$}}\kern-.5\wd0}}

\newcommand{\e}{\varepsilon}

\author{Huy The Nguyen}
\address{School of Mathematical Sciences\\
	Queen Mary University of London\\
	Mile End Road\\
	London E1 4NS}
\email{h.nguyen@qmul.ac.uk}

\author{Artemis Vogiatzi}
\address{Copenhagen Centre for Geometry and Topology\\
	University of Copenhagen\\
	Universitetsparken 5, 2100 Kobenhavn, Danmark}
\email{aav@math.ku.dk}

\begin{document}
\title[]{High Codimension Curve Shortening Flow with Free Boundary}

\begin{abstract}
We study curve shortening flow in high codimension for arcs with free boundary meeting a fixed smooth barrier orthogonally. We prove dilation-invariant curvature and higher-derivative estimates up to the boundary using a Stahl-type localised maximum principle and an adapted cut-off. Using a reflected Gaussian entropy and blow-up analysis, Type I boundary singularities yield a shrinking semicircle model after reflection. Type II blow-ups give a Grim Reaper translator, which is ruled out under a free-boundary entropy bound $<2$. Hence in the low-entropy regime the flow either converges to the orthogonal chord or has only semicircle boundary singularities.
\end{abstract}
\maketitle

\section{Introduction}
In this paper, we study free-boundary curve-shortening flow in higher codimension and classify its singularities. As a first step, we establish curvature and gradient estimates up to the boundary and prove a parabolic compactness theorem ensuring that blow-up limits exist; this is in the same spirit as the boundary regularity theory for Neumann problems developed by Stahl \cite{Stahl1996} and the compactness framework for free-boundary Brakke flows and reflected monotonicity established by Edelen \cite{Edelen2016}. With these tools in hand, we distinguish between Type I and Type II singularities and perform a blow-up procedure in each case, following the general philosophy of the classical analysis for space curves due to Altschuler and Altschuler-Grayson \cite{Altschuler91,AG92}.

Under a Type I assumption, we apply Edelen's reflected monotonicity formula with the recentred reflected heat kernel \cite{Edelen2016} and carry out a parabolic blow-up. We show that, in all codimensions, the resulting boundary tangent flow is a free-boundary self-shrinker in a half-plane, so that the singularity becomes effectively codimension one. In the limit, the boundary line becomes flat, which allows us to reflect across it and obtain a classical interior self-shrinker \cite{Edelen2016}. Under the entropy bound $\operatorname{Ent}<2$, the only possible interior shrinker is the round circle, in view of the classification of homothetically shrinking solutions to curve-shortening flow by Abresch and Langer \cite{AL86}, and hence the boundary model is a half-circle.

At a boundary Type II blow-up point, we rescale at the curvature scale $\lambda_i=\kappa(p_i,t_i)$ around $(p_i,t_i)$ with $t_i\to T$. Edelen's reflected monotonicity yields uniform Gaussian bounds for these rescaled free-boundary flows and therefore subsequential convergence to a smooth eternal free-boundary limit in a half-plane \cite{Edelen2016}. Reflecting after passing to the limit produces an eternal interior solution. Hamilton's Harnack inequality \cite{Hamilton} then forces this interior limit to be a translator. To identify the translator, we either (i) invoke Altschuler's torsion decay mechanism along essential Type II blow-up sequences (so the torsion goes to $0$ along the sequence) \cite{Altschuler91}, or (ii) use the translator identity $\kappa N=V^\perp$, which implies the normal line is parallel and hence the curve is planar. In either case, the reflected limit must be the Grim Reaper translator, which is the canonical planar translator singled out by the Type II theory \cite{AG92,Hamilton}. Undoing the reflection shows that the original boundary tangent flow is one half of the Grim Reaper, meeting the boundary orthogonally.

Finally, since translators have entropy at least $2$, this Type II scenario cannot occur under our entropy bound. Thus, under $\operatorname{Ent}<2$, the only singularity model is the half-circle; otherwise, the flow exists for all time and converges to a geodesic chord with free boundary, as shown by Langford and Zhu \cite{LanZhu}.

\subsection{Curve shortening flow in high codimension}

For curve-shortening flow, we say a one parameter family of curve $\gamma(t)\subset\Omega$ has a geometric free boundary in a fixed barrier $\partial\Omega$ when its endpoints lie on $\partial\Omega$ (i.e. $\partial \gamma \subset \partial\Omega$) and $\gamma$ meets $\partial\Omega$ at right angles at those endpoints. This orthogonal-contact condition is both physically intuitive and mathematically natural in boundary-value geometric evolution problems.

\begin{definition}[Free Boundary Curve Shortening Flow]
Let $\Omega \subset \mathbb{R}^{n+1}$ be a smooth, compact domain with smooth boundary $\partial \Omega$. Consider a one-parameter family of immersed curves
	\begin{align*}
	\gamma(\cdot, t): I \rightarrow \Omega \quad t \in[0, T),
	\end{align*}

where $I=[a,b] \subseteq \mathbb{R}$ is an interval (e.g., $I=S^1$ for closed curves, or a finite interval for curves with endpoints on $\partial \Omega$).

We say the family of curves $\gamma(\cdot, t)$ evolves by Free Boundary Curve Shortening Flow (FBCSF) if it satisfies:
\begin{enumerate}
\item Curve Shortening Flow (interior equation)

At every interior point of the curve, the evolution is governed by
	\begin{align*}
	\frac{\partial \gamma}{\partial t}(p, t)&=\vec{\kappa}(p, t)\\
	\gamma(\cdot,0)&=\gamma_0,
	\end{align*}
for all $(p,t)\in I\times [0,T)$, where $\vec{\kappa}(p, t)$ is the curvature vector at the point $\gamma(p, t)$.
\item Free Boundary Condition (orthogonality)

At boundary intersection points (if any), the curve meets the boundary $\partial \Omega$ orthogonally:
	\begin{align*}
	\langle T(p, t), X\rangle=0, \quad \text{ and } \quad \gamma(\partial I, t) \subset \partial \Omega,
	\end{align*}

for all $(p,t)\in I\times [0,T)$, where $T(p, t)$ is the tangent vector to the curve at $\gamma(p, t)$ and $X \in T_{\gamma(p, t)}\partial \Omega$.
\end{enumerate}
\end{definition}

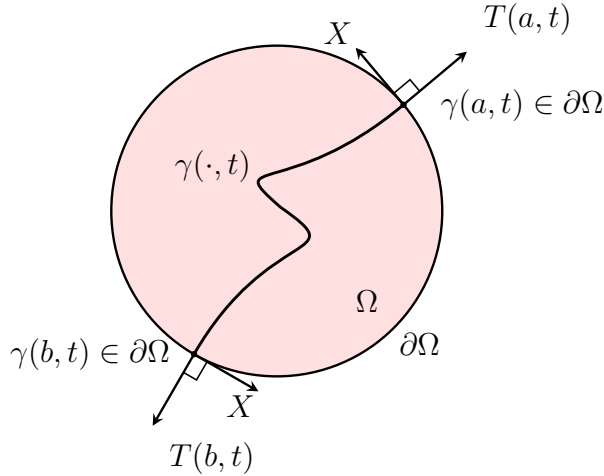
\begin{center}

\begin{tikzpicture}[scale=1.0, >=stealth]
  \def\R{2.2}  \def\angP{40}  \def\angQ{-120}  \def\Tlen{1.1} 
  \fill[red!12] (0,0) circle (\R);
 \draw[line width=0.9pt] (0,0) circle (\R) node[below right=42pt] {$\partial \Omega$};
 \node at (1.2,-1.2) {$\Omega$};

  \coordinate (a) at ({\R*cos(\angP)},{\R*sin(\angP)});
 \coordinate (b) at ({\R*cos(\angQ)},{\R*sin(\angQ)});

  \coordinate (radP) at ({cos(\angP)},{sin(\angP)});
 \coordinate (tanP) at ({-sin(\angP)},{cos(\angP)});
 \coordinate (radQ) at ({cos(\angQ)},{sin(\angQ)});
 \coordinate (tanQ) at ({-sin(\angQ)},{cos(\angQ)});

  \coordinate (m) at ($(0,0)+(0.05,0.05)$);
 \draw[line width=1.1pt]
 (a)
.. controls ($ (a) - 2.0*(radP) $) and ($ (m) + (-0.80,0.70) $).. (m)
.. controls ($ (m) + ( 1.0,-0.70) $) and ($ (b) - 2.0*(radQ) $).. (b)
 node[pos=-0.05, above left=2pt] {$\gamma(\cdot,t)$};

  \draw[->, line width=0.9pt] (a) -- ($ (a) + \Tlen*(radP) $) node[above right=2pt] {$T(a,t)$};
 \draw[->, line width=0.9pt] (a) -- ($ (a) + 0.9*\Tlen*(tanP) $) node[above left=-3pt] {$X$};

  \draw[->, line width=0.9pt] (b) -- ($ (b) + \Tlen*(radQ) $) node[below right=2pt] {$T(b,t)$};
 \draw[->, line width=0.9pt] (b) -- ($ (b) + 0.9*\Tlen*(tanQ) $) node[below left=-3pt] {$X$};

   \draw[line width=0.6pt]
 ($ (a) + 0.28*(radP) $) --
 ($ (a) + 0.28*(radP) + 0.20*(tanP) $) --
 ($ (a) + 0.20*(tanP) $);
  \draw[line width=0.6pt]
 ($ (b) + 0.28*(radQ) $) --
 ($ (b) + 0.28*(radQ) + 0.20*(tanQ) $) --
 ($ (b) + 0.20*(tanQ) $);

  \fill (a) circle (1.2pt) node[ right=10pt] {$\gamma(a,t)\in\partial \Omega$};
 \fill (b) circle (1.2pt) node[ left=5pt] {$\gamma(b,t)\in\partial \Omega$};
\end{tikzpicture}
\end{center}

\begin{theorem}[Long--time behaviour under low free--boundary entropy]
\label{thm_FBCSF-classification}
Let $\Omega\subset\mathbb R^{n+1}$ ($n\ge 2$) be a bounded, strictly convex domain with $C^2$ boundary.
Let $\{\gamma_t\}_{t\in[0,T)}$ be a maximal {\em free-boundary curve shortening flow} in $\Omega$, starting from a properly embedded $C^2$-arc $\gamma_0$ that meets
$\partial\Omega$ orthogonally.
Assume the {\em reflected Gaussian entropy} of the flow satisfies
	\begin{align*}
	 \operatorname{Ent}_{\partial\Omega}\bigl[\{\gamma_t\}\bigr] < 2.
	\end{align*}
\smallskip
Then exactly one of the following alternatives occurs:

\begin{enumerate}
\item[\textup{(a)}] (\textbf{Infinite-time existence})
 $T=\infty$, and the family $\gamma_t$ converges smoothly, as $t\to\infty$, to a unique straight chord $\overline{pq}\subset\overline\Omega$. In particular, the limit chord meets $\partial\Omega$ orthogonally at $p,q\in\partial\Omega$.

\smallskip
\item[\textup{(b)}] (\textbf{Boundary-type singularity})
$T<\infty$, the diameter of $\gamma_t$ tends to $0$, and the whole curve converges uniformly to a single boundary point $z\in\partial\Omega$. Moreover, writing
 	\begin{align*}
	 \widetilde\gamma_t :=
 \frac{\gamma_t - z}{\sqrt{2(T-t)}} \subset T_z\Omega ,
 	\end{align*}
the rescaled curves $\widetilde\gamma_t$ converge smoothly, as $t\to T$, to the unit semicircle in the tangent half-space $T_z\Omega\simeq\mathbb R^{n+1}_+$.
\end{enumerate}
\end{theorem}

\subsection{Acknowledgements} The first author was financially supported by the EPSRC through the grant ``Geometric Flows and the Dynamics of Phase Transitions" (EP/Y017862/1). The second author is financially supported by the Danish National Research Foundation through the Copenhagen Centre for Geometry and Topology (DNRF151).

\section{Evolution Equations}

We denote by $s $ the arc-length parameter, and use:
\begin{itemize}
 \item $\partial_s $: derivative with respect to arc length,
 \item $T = \partial_s \gamma $: unit tangent,
 \item $N $: principal normal,
 \item $\kappa $: scalar curvature,
 \item $\tau_i $: torsions.
\end{itemize}

\begin{definition}
Let $\gamma: I \times [0,T) \to \mathbb{R}^{n+1} $ be a smooth family of immersed curves. We say that $\gamma$ evolves under Free Boundary Curve Shortening Flow (FBCSF) inside a smooth domain $\Omega \subset \mathbb{R}^{n+1} $ if $\gamma$ satisfies the following partial differential equation
	\begin{align}\label{eqn_FreeBoundaryCSF}
	\partial_t \gamma = \kappa N,
	\end{align}
where $\kappa $ is the curvature and $N $ is the principal normal vector. At the boundary $\gamma(\cdot,t) \in \partial \Omega $, the curve satisfies the free boundary condition which states that the curve meets the boundary orthogonally - that is if $T$ is the tangent to the curve and $X$ is any tangent vector to $\partial\Omega$ then
	\begin{align*}
	\langle T, X \rangle = 0.
	\end{align*}
\end{definition}

\begin{remark}
At a boundary point $p\in \partial I$ we have $\gamma(p,t)\in \partial\Omega$ and the free
boundary condition says that $T(p,t)\perp T_{\gamma(p,t)}\partial\Omega$. Since $T_{\gamma(p,t)}\partial\Omega$ is a
hyperplane with unit normal $\nu_{\partial\Omega}(\gamma(p,t))$, it follows that
	\begin{align*}
	T(p,t)=\pm \nu_{\partial\Omega}(\gamma(p,t)).
	\end{align*}
This identity holds only at the endpoints (not along the interior of the curve).
\end{remark}

\begin{theorem}[Interior evolution equations for curve shortening in $\mathbb{R}^{n+1}$]
\label{thm_interior-evolution}
Let $\gamma(\cdot,t)$ be a smooth space curve evolving by curve shortening
$\partial_t \gamma=\kappa N$. Write $s$ for arclength at time $t$, and let
$(T,N,B)$ be the Frenet frame with curvature $\kappa$ and first and second torsion $\tau_1,\tau_2$.
Then, pointwise in the interior (and wherever $\kappa\neq 0$ for the torsion
formula), the following hold:
	\begin{align}
	\text{\emph{Tangent vector:}}\quad
&\partial_t T = \partial_s^2 T + \kappa^2 T = \kappa_sN + \kappa\tau_1 B_1,
\label{eqn_evol-T}\\
	\text{\emph{Curvature vector:}}\quad
&\partial_t(\kappa N)=-\kappa \frac{\partial \kappa}{\partial s} T+\left(\frac{\partial^2 \kappa}{\partial s^2}+\kappa^3-\tau_1^2 \kappa\right) N\nonumber\\
	&\quad+\left(\kappa \frac{\partial \tau_1}{\partial s}+2 \tau_1 \frac{\partial \kappa}{\partial s}\right) B_1+\kappa \tau_1 \tau_2 B_2,
\label{eqn_evol-kN}\\
	\text{\emph{Curvature scalar:}}\quad
&\partial_t \kappa = \partial_s^2 \kappa + \kappa^3 - \tau_1^2\kappa,
\label{eqn_evol-kappa}\\
	\text{\emph{Curvature squared:}}\quad
&\partial_t(\kappa^2) = \partial_s^2(\kappa^2) - 2(\kappa_s)^2 + 2\kappa^4 - 2\tau_1^2\kappa^2,
\label{eqn_evol-kappa2}\\
	\text{\emph{Torsion:}}\quad
&\partial_t \tau_1=\frac{\partial^2 \tau_1}{\partial s^2}+\frac{2}{\kappa} \frac{\partial \kappa}{\partial s} \frac{\partial \tau_1}{\partial s}\nonumber\\
	&\quad+\frac{2 \tau_1}{\kappa}\left(\frac{\partial^2 \kappa}{\partial s^2}-\frac{1}{\kappa}\left(\frac{\partial \kappa}{\partial s}\right)^2+\kappa^3\right)-\tau_1 \tau_2^2,\qquad \kappa\neq0.
\label{eqn_evol-tau}
	\end{align}
Equivalently, with $a:=(\log \kappa)_s=\kappa_s/\kappa$ and $a_s=(\log \kappa)_{ss}$,
	\begin{align}
	&(\partial_t-\partial_s^2) \tau_1=2 \frac{\partial}{\partial s}\left(a \tau_1\right)+2 \tau_1 \kappa^2-\tau_1 \tau_2^2,
\label{eqn_evol-tau-alt}\\
	&(\partial_t-\partial_s^2)\left(\tau_1^2\right)=2 a \frac{\partial}{\partial s}\left(\tau_1^2\right)-2\left(\frac{\partial \tau_1}{\partial s}\right)^2+4 \tau_1^2 \frac{\partial a}{\partial s}+4 \tau_1^2 \kappa^2-2 \tau_1^2 \tau_2^2.
\label{eqn_evol-tau2}
	\end{align}
\end{theorem}

\begin{remark}
\textbf{Planar case.} If the curve lies in a plane (so $\tau_1\equiv 0$), these reduce to
$\partial_t T=\partial_s^2T+\kappa^2T$,
$\partial_t\kappa=\partial_s^2\kappa+\kappa^3$ and
$\partial_t(\kappa^2)=\partial_s^2(\kappa^2)-2(\kappa_s)^2+2\kappa^4$.

\end{remark}

\section{Short Time Existence}
Following \cite{LanZhu}, we say that $\left\{\gamma_t\right\}_{t \in[0, T)} \subset \Omega \subset \mathbb{R}^{n+1}$ is a (classical) free-boundary curve shortening flow if it evolves by $\partial_t \gamma=\vec{\kappa}$ and satisfies the geometric free-boundary conditions $\gamma(\partial I, t) \subset \partial \Omega$ together with orthogonal contact (equivalently, the conormal of $\gamma_t$ agrees with $\nu_{\partial\Omega}$ at $\partial \gamma_t$). Short-time existence (and uniqueness, up to reparametrisation) is obtained by a straightforward high-codimension adaptation of \cite{Stahl1996} in the Neumann/free-boundary existence theory: one works in a tubular neighbourhood of the initial curve and writes $\gamma_t$ as a normal graph over $\gamma_0$. In higher codimension the unknown becomes $\mathbb{R}^n$-valued, and the orthogonality condition becomes a nonlinear oblique (Neumann-type) boundary condition for this quasilinear parabolic system. Standard parabolic boundary-value theory then yields a unique smooth solution on $[0, \delta)$ and hence a maximal smooth existence time $T \in(0, \infty]$ with the usual continuation criterion (loss of smoothness can only occur through curvature blow-up). Finally, we also view any such smooth free-boundary solution as an instance of \cite{Edelen2016} in the free-boundary Brakke flow framework (a smooth free-boundary flow induces a free-boundary Brakke flow and Edelen proves existence/smooth short-time existence from smooth embedded initial data), which we use later for compactness, reflected monotonicity and tangent flow arguments in arbitrary codimension.

\begin{theorem}(short-time existence) There exists $\delta>0$ and a unique smooth family of immersions
	\begin{align*}
	\gamma \in C^{\infty}(I \times(0, \delta]) \cap C^{2+\alpha, 1+\alpha / 2}(I \times[0, \delta]),
	\end{align*}
such that $\gamma_t:=\gamma(I, t)$ evolves by curve shortening flow
	\begin{align*}
	\partial_t \gamma=\vec{\kappa}
	\end{align*}
and satisfies the free-boundary conditions for all $t \in[0, \delta]$
	\begin{align*}
	\gamma(\partial I, t) \subset \partial\Omega, \quad \gamma_t \perp \partial\Omega \text{ along } \partial I.
	\end{align*}
\end{theorem}

\begin{proof}
We adapt Stahl's reduction to a strictly parabolic quasilinear boundary value problem, with the only
high-codimension change being that the unknown is $\mathbb{R}^n$-valued and the boundary condition is an
$\mathbb{R}^n$-valued oblique operator. The key geometric point is to build a tubular chart in which the endpoint
constraint $\gamma(\partial I,t)\subset\partial\Omega$ is automatic, so that only the orthogonality remains as a
first-order oblique condition.

{\it Step 1}: Boundary-adapted tubular coordinates (Stahl's geometric input, in codimension $n$).\\
	 Fix a smooth orthonormal frame $\{\nu_a\}_{a=1}^n$ of the normal bundle $N\gamma_0$
along $I$. Choose a small neighbourhood $\mathcal{U}$ of $\gamma_0(I)$ in $\mathbb{R}^{n+1}$.
We construct smooth vector fields $\xi_1,\dots,\xi_n$ on $\mathcal{U}$ such that
	\begin{align*}
	\xi_a|_{\gamma_0}=\nu_a,\qquad \xi_a|_{\partial\Omega\cap\mathcal{U}}\in T\partial\Omega,\qquad \langle \xi_a,\xi_b\rangle=\delta_{ab}
\ \text{ on }\mathcal{U}.
	\end{align*}
A concrete construction is as follows: extend $\nu_a$ smoothly off $\gamma_0$ to $\mathcal{U}$ and on $\partial\Omega\cap\mathcal{U}$
project these extensions onto $T\partial\Omega$. Then (after shrinking $\mathcal{U}$ if needed) apply a smooth Gram--Schmidt
procedure to obtain an orthonormal $n$-frame with the stated properties.

Let $\Phi_a(\cdot,s)$ denote the flow of $\xi_a$. For $\e>0$ sufficiently small, define
	\begin{align*}
	\Psi:I\times B_\e^n\to\mathbb{R}^{n+1},\qquad
\Psi(x,w)=\Phi_n(\cdot,w^n)\circ\cdots\circ\Phi_1(\cdot,w^1)\bigl(\gamma_0(x)\bigr),
	\end{align*}
where $B_\e^n\subset\mathbb{R}^n$ is the radius-$\e$ ball. For $\e$ small, $\Psi$ is a
diffeomorphism onto its image (by the inverse function theorem, using that $D_w\Psi(x,0)$ maps the standard basis
of $\mathbb{R}^n$ to $\{\xi_a(\gamma_0(x))\}$). Since each $\xi_a$ is tangent to $\partial\Omega$ along $\partial\Omega$, each flow
$\Phi_a$ preserves $\partial\Omega$ locally; therefore, for $x\in\partial I$ and all $w\in B_\e^n$ we have
	\begin{align*}
	\gamma_0(x)\in\partial\Omega \implies \Psi(x,w)\in\partial\Omega.
	\end{align*}
Hence, in the chart $\Psi$, the endpoint constraint $\widetilde\gamma(\partial I,t)\subset\partial\Omega$ will hold
automatically.

{\it Step 2}: Graph gauge and reduction to a strictly parabolic quasilinear system.\\
	Seek the evolving curve in the form
	\begin{align*}
	\widetilde{\gamma}(x,t)=\Psi(x,w(x,t)),
\qquad
w:I\times[0,\delta]\to B_\e^n\subset\mathbb{R}^n,
\qquad
w(\cdot,0)=0.
	\end{align*}
Fix $x$ as a parameter on $I$ (for instance the arclength parameter of $\gamma_0$). Writing the curvature vector of
$\widetilde\gamma$ in these coordinates yields a quasilinear system for $w$ of the form
	\begin{align}\label{derofw}
	\partial_t w=A\left(x, w, \partial_x w\right) \partial_{x x} w+B\left(x, w, \partial_x w\right),
	\end{align}
where $A(x,w,p)$ is an $n\times n$ matrix depending smoothly on $(x,w,p)$, and
	\begin{align*}
	A(x,0,0)=\mathrm{Id}_{\mathbb{R}^n}.
	\end{align*}
In particular, $A$ is uniformly positive definite whenever $|w|+|\partial_x w|$ is sufficiently small. Since
$w(\cdot,0)=0$, strict parabolicity holds for short time along any solution with small $C^1$ norm.

{\it Step 3}: Boundary conditions become an $\mathbb{R}^n$-valued oblique operator of full rank.\\
	By Step 1, for every $t$ we automatically have $\widetilde\gamma(\partial I,t)\subset\partial\Omega$. Thus the only remaining free-boundary condition is orthogonality at the endpoints
	\begin{align*}
	T_{\widetilde\gamma_t}\perp T\partial\Omega \ \text{ along }\partial I.
	\end{align*}
Fix along $\partial\Omega\cap\mathcal{U}$ an orthonormal frame $\{e_\beta\}_{\beta=1}^n$ of $T\partial\Omega$. Define the boundary
operator $\mathcal{B}=(\mathcal{B}_1,\dots,\mathcal{B}_n)$ by
	\begin{align}\label{B}
	\mathcal{B}_\beta\left(x,w,\partial_x w\right)
=
\left\langle \partial_x\widetilde\gamma(x,t),\, e_\beta(\widetilde\gamma(x,t))\right\rangle,
\qquad
\beta=1,\dots,n,
\qquad x\in\partial I.
	\end{align}
Then $T_{\widetilde\gamma_t}\perp T\partial\Omega$ along $\partial I$ is exactly $\mathcal{B}(x,w,\partial_x w)=0$ on
$\partial I\times[0,\delta]$.

We verify obliqueness at $w=0$. Differentiating $\widetilde\gamma(x,t)=\Psi(x,w(x,t))$ in $x$ gives
	\begin{align*}
	\partial_x\widetilde\gamma=\partial_x\Psi(x,w)+D_w\Psi(x,w)\partial_x w.
	\end{align*}
At $w=0$, $D_w\Psi(x,0)[v]=\sum_{a=1}^n v^a\xi_a(\gamma_0(x))$. For $x\in\partial I$, the vectors
$\xi_a(\gamma_0(x))$ lie in $T_{\gamma_0(x)}\partial\Omega$. Since $\gamma_0$ meets $\partial\Omega$ orthogonally, $T\gamma_0(x)$ is
parallel to $\nu_{\partial\Omega}$, so $\dim T_{\gamma_0(x)}\partial\Omega=n$ and $\{\xi_a(\gamma_0(x))\}_{a=1}^n$ is an orthonormal
basis of $T_{\gamma_0(x)}\partial\Omega$. Therefore the linearisation of \eqref{B} in the $\partial_x w$ variable at $(w,p)=(0,0)$ is an isomorphism
	\begin{align*}
	D_{\partial_x w}\mathcal{B}(x,0,0)[v]=\left(\left\langle \sum_{a=1}^n v^a\xi_a(\gamma_0(x)),\, e_\beta(\gamma_0(x))\right\rangle\right)_{\beta=1}^n
	\end{align*}
and the resulting $n\times n$ matrix is invertible because both $\{\xi_a(\gamma_0(x))\}$ and $\{e_\beta(\gamma_0(x))\}$
are bases of $T_{\gamma_0(x)}\partial\Omega$. Hence $\mathcal{B}$ is an oblique boundary operator of full rank at $w=0$, with
linearisation of the schematic form
	\begin{align*}
	D\mathcal{B}(x,0,0)[v]=\Gamma(x)\partial_x v+\Lambda(x)v,\qquad \Gamma(x)\ \text{ invertible for }x\in\partial I.
	\end{align*}
The order $0$ compatibility holds because $w(\cdot,0)=0$ implies $\widetilde\gamma(\cdot,0)=\gamma_0$, and $\gamma_0$
satisfies the free-boundary conditions.

{\it Step 4}: Solve by linear theory plus the implicit function theorem.\\
	Linearising \eqref{derofw} and \eqref{B} at $w=0$ gives a strictly parabolic linear system on an interval with an oblique
boundary condition whose $\partial_x$ coefficient is invertible. In this one-dimensional setting, obliqueness implies the
complementing (Lopatinskii--Shapiro) condition, so standard Schauder theory for parabolic systems with oblique boundary
conditions yields existence, uniqueness, and estimates in $C^{2+\alpha,1+\alpha/2}$ for the linear problem. Therefore the
linearisation of the nonlinear operator $(\text{ PDE},\text{ BC},\text{ IC})$ at $w=0$ is an isomorphism between the relevant
H\"older spaces. The implicit function theorem then gives, for $\delta>0$ sufficiently small, a unique solution
	\begin{align*}
	w\in C^{2+\alpha,1+\alpha/2}(I\times[0,\delta]),
\qquad w(\cdot,0)=0,
	\end{align*}
to the nonlinear boundary value problem \eqref{derofw}, \eqref{B}. Parabolic bootstrapping implies $w$ is smooth for $t>0$,
hence $\widetilde\gamma$ is smooth for $t>0$.

{\it Step 5}: DeTurck reparametrisation to obtain the geometric flow and uniqueness up to reparametrisation.\\
	The gauge $\widetilde\gamma(x,t)=\Psi(x,w(x,t))$ fixes a parametrisation, so $\partial_t\widetilde\gamma$ may differ from the
geometric law $\vec\kappa$ by a tangential term. Let $V^{\top}$ be the tangential component of
$\partial_t\widetilde\gamma-\vec\kappa$. Solve the ODE for a family of diffeomorphisms $\varphi_t:I\to I$ with
$\varphi_t(\partial I)=\partial I$ so that the pullback curve
	\begin{align*}
	\gamma(\cdot,t):=\widetilde{\gamma}(\varphi_t(\cdot),t)
	\end{align*}
has vanishing tangential error. Then $\gamma$ satisfies $\partial_t\gamma=\vec\kappa$ and the free-boundary conditions. Uniqueness up to reparametrisation follows from uniqueness for the strictly parabolic boundary value problem for $w$ and uniqueness for the ODE defining $\varphi_t$.
\end{proof}

\section{Dilation-Invariant Estimates}

Let $M_t = \sup_{\gamma_t} \kappa^2 $. Then, for short time intervals $[t_0, t_0 + c/M_{t_0}] $, we have
	\begin{align*}
	|\partial_s^m T| \leq C_m M_{t_0}^{m/2},
	\end{align*}
for constants $C_m $ depending only on $m $, valid in the interior and near the boundary using suitable reflection techniques.

\subsection{Free-boundary adaptation of the dilation-invariant estimates}
\label{sec:FBCSF-adapt}

We show that Altschuler's interior derivative bounds remain valid up to a free boundary by combining Altschuler's interior computations with Stahl's localised maximum principle.

Fix $t_0$ and set
	\begin{align*}
	M_{t_0}:=\sup_{\gamma_{t_0}}\kappa^2,\qquad
I_{t_0}:=\Bigl[t_0,t_0+\frac{c}{M_{t_0}}\Bigr].
	\end{align*}
The goal is to prove, for every integer $m\ge 1$,
	\begin{align}\label{eqn_FBCSF-derivative}
	|\partial_s^m T|^2\le \frac{C_m M_{t_0}}{(t-t_0)^{m-1}},\qquad\text{ for all }t\in I_{t_0},
	\end{align}
with constants $C_m$ independent of $t_0$; compare Altschuler \cite[Part II, \S 3]{Altschuler91}.

\medskip

\begin{theorem}[cf.\cite{Stahl1996}, Theorem 3.1 and Section 6.3](Stahl--type localised maximum principle up to a free boundary).\label{thm_Stahl-localMP}
Let $\{\mathcal{M}_t\}_{t\in[t_0,t_0+\theta]}$ be a smooth solution of mean curvature flow in $\mathbb R^{n+1}$
with free boundary on a fixed smooth hypersurface $\partial\Omega$, satisfying
$\langle \nu_{\mathcal{M}_t},\nu_{\partial\Omega}\rangle=0$ along $\partial \mathcal{M}_t$. Assume $\partial\Omega$ has
$\|\mathrm{II}_{\partial\Omega}\|_{L^\infty}\le K$ and admits a tubular neighbourhood of radius $\rho>0$.

Fix $(p_0,t_0)\in \overline{\mathcal{M}_{t_0}}\times\{t_0\}$ and $r\in(0,\rho/8]$. Then there exist
$c_0=c_0(K,\rho)>0$, $C=C(K,\rho)\ge 1$, and a spacetime cut-off
	\begin{align*}
	\eta=\eta_{r,\theta}:\ \bigcup_{t\in[t_0,t_0+\theta]}M_t\times\{t\}\longrightarrow[0,1],
	\end{align*}
such that $\eta\equiv 0$ on the parabolic boundary of $Q_{r,\theta}(p_0,t_0)$, $\eta\ge \tfrac12$ on
$Q_{r/2,\theta/2}(p_0,t_0)$, and along $\partial \mathcal{M}_t$ one has the strict outward conormal inequality
	\begin{align}\label{eqn_eta-oblique}
	\partial_{\mu}\eta\le -c_0\eta.
	\end{align}
Moreover, if $Q\ge 0$ is $C^{2,1}$ on $Q_{r,\theta}(p_0,t_0)$ and satisfies
	\begin{align}\label{eqn_Stahl-struct}
	(\partial_t-\Delta)Q\le a_1 Q^2+a_2 Q+a_3,\qquad
\partial_{\mu}Q\le b_1 Q+b_2\quad\text{ on }\partial \mathcal{M}_t,
	\end{align}
for non-negative constants $a_i,b_i$ independent of $r,\theta$, then for $\Phi:=\eta^2 Q$ any large
parabolic maximum of $\Phi$ on $Q_{r,\theta}$ occurs either at $t=t_0$ or at an interior point.

Consequently,
	\begin{align}\label{eqn_Stahl-conclusion}
	\sup_{Q_{r/2,\theta/2}(p_0,t_0)} Q\le C'\left(\frac{1}{r^2}+\frac{1}{\theta}+1\right)
+\sup_{\{t=t_0\}\cap Q_{r,\theta}(p_0,t_0)} Q,
	\end{align}
for a constant $C'=C'(K,\rho,a_i,b_i)$.
\end{theorem}
\begin{remark}
Stahl does not state \eqref{eqn_Stahl-conclusion} verbatim; it follows by applying
the weak/strong maximum principle to $\Phi=\eta^2 Q$, with $\eta$ chosen via the
shifted-paraboloid construction in \cite[Section 6.3]{Stahl1996} to preclude large
boundary maxima.
\end{remark}
\medskip

The interior evolution identities for curvature and its derivatives are unchanged by the presence
of a free boundary. The only additional input is a boundary control for the conormal derivative of
the relevant quantities.

\begin{lemma}[Evolution and structural inequalities]\label{lem_evol-struct}
Let $\gamma_t\subset\mathbb R^{n+1}$ be a smooth curve evolving by CSF $\partial_t\gamma=\kappa N$,
with arclength parameter $s$ and first torsion $\tau_1$. Then,
\begin{enumerate}
\item[(i)] The curvature evolves by
	\begin{align}\label{eqn_evol-kappa}
	\partial_t\kappa=\partial_s^2\kappa+\kappa^3-\tau_1^2\kappa.
	\end{align}
Consequently,
	\begin{align}\label{eqn_evol-kappa2}
	(\partial_t-\partial_s^2)\kappa^2=-2(\partial_s\kappa)^2+2\kappa^4-2\tau_1^2\kappa^2\le 2(\kappa^2)^2.
	\end{align}
\item[(ii)] The commutator satisfies
	\begin{align}\label{eqn_commute-ts}
	[\partial_t,\partial_s]=\kappa^2\,\partial_s.
	\end{align}
\item[(iii)] For each $m\ge 1$ there are constants $A_m,B_m$ depending only on $m$ and on bounds for
lower order quantities such that, writing $Q_m:=|\partial_s^m\kappa|^2$, one has a structural inequality
of the form
	\begin{align}\label{eqn_struct-Qm}
	(\partial_t-\partial_s^2)Q_m\le A_m Q_m^2+B_m Q_m
	\end{align}
on any spacetime region where the lower order derivatives entering the coefficients are controlled.
\end{enumerate}
\end{lemma}

\begin{lemma}[Endpoint control]\label{lem_endpoint-control}
Let $\gamma_t\subset\mathbb R^{n+1}$ solve curve shortening flow with free boundary on a smooth
surface $\partial\Omega$, meeting $\partial\Omega$ orthogonally along $\partial\gamma_t$, and assume
$\|\mathrm{II}_{\partial\Omega}\|_{L^\infty}\le K$. Then on $\partial\gamma_t$,
	\begin{align}\label{eqn_endpoint-kappa}
	|\partial_{\mu}\kappa|\le K|\kappa|,\qquad |\tau_1|\le K,
	\end{align}
where $\mu$ is the outward conormal to $\partial\gamma_t$ in $\gamma_t$.
Consequently,
	\begin{align}\label{eqn_endpoint-kappa2}
	\partial_{\mu}(\kappa^2)\le 2K\kappa^2,\qquad\text{ on }\partial\gamma_t.
	\end{align}
\end{lemma}

\begin{proof}
Fix an endpoint $p\in\partial\gamma_t$. The free boundary condition gives
	\begin{align*}
	T(p,t)=\e\,\nu_{\partial\Omega}(p),\qquad \e=\pm 1,
	\end{align*}
so $N(p,t)$ and $B(p,t)$ lie in $T_p\partial\Omega$. The endpoint velocity is
$\partial_t\gamma=\kappa N\in T_p\partial\Omega$.

Differentiate $T=\e\,\nu_{\partial\Omega}$ in time at $p$:
	\begin{align*}
	\partial_t T=\e\, D_{\partial_t\gamma}\nu_{\partial\Omega}=-\e\, S(\partial_t\gamma)=-\e\,\kappa\, S(N),
	\end{align*}
where $S$ is the shape operator of $\partial\Omega$. Decompose
	\begin{align*}
	S(N)=\mathrm{II}_{\partial\Omega}(N,N)\,N+\mathrm{II}_{\partial\Omega}(N,B)\,B,
	\end{align*}
and use the Frenet--Serret evolution identity
	\begin{align*}
	\partial_t T=(\partial_s\kappa)\,N+\kappa\tau_1\,B.
	\end{align*}
Comparing $N$ and $B$ components yields
	\begin{align}\label{sffK}
	\partial_s\kappa=-\e\,\kappa\,\mathrm{II}_{\partial\Omega}(N,N),\qquad \tau_1=-\e\,\mathrm{II}_{\partial\Omega}(N,B).
	\end{align}
Since $\partial_{\mu}=\pm\partial_s$ at an endpoint, \eqref{eqn_endpoint-kappa} follows and
\eqref{eqn_endpoint-kappa2} is immediate from $\partial_{\mu}(\kappa^2)=2\kappa\,\partial_{\mu}\kappa$.
\end{proof}

\begin{lemma}[Boundary reduction for $T$-derivatives]\label{lem_boundary-reduction}
At an endpoint, $\partial_\mu=\pm\partial_s$. Moreover, for each $m\ge 1$ there exist universal
polynomials $P_m,\widetilde P_m$, such that
	\begin{align*}
	\partial_s^mT=P_m\big(\kappa,\tau_1,\partial_s\kappa,\partial_s\tau_1,\dots,\partial_s^{m-1}\kappa \partial_s^{m-2}\tau_1\big),
	\end{align*}
	\begin{align*}
	\partial_\mu(|\partial_s^mT|^2)=\pm 2\langle \partial_s^mT,\partial_s^{m+1}T\rangle =\widetilde P_m\big(\kappa,\tau_1,\dots,\partial_s^m\kappa,\partial_s^{m-1}\tau_1\big).
	\end{align*}
Consequently, on any spacetime region where the lower order derivatives appearing in
$\widetilde P_m$ are bounded, one has a boundary inequality of the form
	\begin{align*}
	\partial_\mu Q_m\le b_1 Q_m+b_2,
	\end{align*}
with $b_1,b_2$ depending only on $m$, $K$, and those lower order bounds.
\end{lemma}

We first record the scale-invariant curvature bound up to the boundary.

\begin{lemma}[Local pointwise curvature estimate up to a free boundary]\label{lem_local-kappa2}
Let $\gamma_t$ solve free-boundary curve shortening flow on $\partial\Omega$ with
$\|\mathrm{II}_{\partial\Omega}\|_{L^\infty}\le K$ and tubular radius $\rho>0$. Fix $(p_0,t_0)\in\overline{\gamma_{t_0}}$,
choose $r\in(0,\rho/8]$ and $\theta\le r^2$, and let $\eta$ be the cutoff from
Theorem \ref{thm_Stahl-localMP}. Then
	\begin{align}\label{eqn_local-kappa2}
	\sup_{Q_{r/2,\theta/2}(p_0,t_0)}\kappa^2 \le C(K,\rho)\left(\frac{1}{r^2}+\frac{1}{\theta}\right) +\sup_{\{t=t_0\}\cap Q_{r,\theta}(p_0,t_0)}\kappa^2.
	\end{align}
\end{lemma}

\begin{proof}
By Lemma \ref{lem_evol-struct}(i), $\kappa^2$ satisfies the interior inequality
	\begin{align*}
	(\partial_t-\partial_s^2)\kappa^2\le 2\kappa^4=2(\kappa^2)^2.
	\end{align*}
By Lemma \ref{lem_endpoint-control}, $\partial_{\mu}(\kappa^2)\le 2K\kappa^2$ on $\partial\gamma_t$.
Thus \eqref{eqn_Stahl-struct} holds for $Q=\kappa^2$ with $a_1=2$, $a_2=a_3=0$, $b_1=2K$, $b_2=0$.
Applying Theorem \ref{thm_Stahl-localMP} gives \eqref{eqn_local-kappa2}.
\end{proof}

\medskip

We now explain the derivative bounds. Altschuler constructs, for each $m\ge 1$, a non-negative
quantity $Q_m$ depending on $\partial_s^j T$, for $1\le j\le m$, which is scale-invariant and satisfies
an interior differential inequality of the form \eqref{eqn_Stahl-struct} on a time interval of length proportional up to a constant to
$ M_{t_0}^{-1}$; see \cite[Part II, \S 3]{Altschuler91}. The free-boundary adaptation consists
only in verifying the boundary inequality in \eqref{eqn_Stahl-struct}, after which Stahl's principle
yields the same local bounds up to the boundary.

\begin{proof}[Proof of Equation \eqref{eqn_FBCSF-derivative}]
Fix $m\ge 1$ and let $Q_m$ be Altschuler's scale-invariant quantity controlling $|\partial_s^m T|$
as in \cite[Part II, \S 3]{Altschuler91}. On the interior of the curve, Altschuler proves an estimate
of the form
	\begin{align}\label{eqn_Altschuler-interior}
	(\partial_t-\partial_s^2)Q_m\le a_1 Q_m^2+a_2 Q_m+a_3
	\end{align}
on a time interval $[t_0,t_0+c/M_{t_0}]$, with constants $a_i$ depending only on $m$.

It remains to control the endpoint conormal derivative. Since $\partial_{\mu}=\pm\partial_s$ at an
endpoint, $\partial_{\mu}Q_m$ can be expanded in terms of $\partial_s^{m+1}T$ and lower derivatives.
Using the endpoint relations in Lemma \ref{lem_endpoint-control} (in particular
$|\partial_{\mu}\kappa|\le K|\kappa|$ and $|\tau_1|\le K$ at $\partial\gamma_t$), one obtains a boundary
inequality of the form
	\begin{align}\label{eqn_Altschuler-boundary}
	\partial_{\mu}Q_m\le b_1 Q_m+b_2\qquad\text{ on }\partial\gamma_t,
	\end{align}
where $b_1,b_2$ depend only on $m$ and $K$. Indeed, from Lemma \ref{lem_endpoint-control}, we have the following. Since,
	\begin{align*}
\partial_s^2 T=-\kappa^2 T+\e\, D_{\partial_t\gamma}\nu_{\partial\Omega} \ \ \text{ and} \ \ \partial_s^2 T=-\kappa^2 T+\partial_s \kappa N+\kappa\tau_1 B,
	\end{align*}
we have
	\begin{align}\label{firstT}
	\partial_\mu\partial_t T&= -\e (\partial_\mu \kappa) \mathrm{II}_{\partial\Omega} (N,N)N-\e(\partial_\mu\kappa)\mathrm{II}_{\partial\Omega} (N,B)B-\e\kappa\nabla_{\partial_t \gamma} \mathrm{II}_{\partial\Omega} (N,N)N\nonumber\\
	&-\e\kappa\mathrm{II}_{\partial\Omega} (N,N)\partial_\mu N -\e\kappa \nabla_{\partial_t \gamma} \mathrm{II}_{\partial\Omega} (N,B)B-\e\kappa \mathrm{II}_{\partial\Omega} (N,B)\partial_\mu B
	\end{align}
and
	\begin{align}\label{secondT}
\partial_\mu\partial_t T=\partial_\mu\partial_s \kappa N+ \partial_s \kappa\partial_\mu N +(\partial_\mu \kappa) \tau_1 B+\kappa (\partial_\mu \tau_1) B + \kappa \tau_1 \partial_\mu B.
	\end{align}
Also, since $\partial_s N=-\kappa T+\tau_1 B$, $\partial_s B=-\tau_1 N$ and $\partial_\mu=\pm \partial_s$, \eqref{firstT} becomes
	\begin{align}\label{firstTs}
	\partial_\mu\partial_t T=\pm\partial_s\partial_t T&= \mp\e (\partial_s \kappa) \mathrm{II}_{\partial\Omega} (N,N)N\mp\e(\partial_s\kappa)\mathrm{II}_{\partial\Omega} (N,B)B-\e\kappa\nabla_{\partial_t \gamma} \mathrm{II}_{\partial\Omega} (N,N)N\nonumber\\
	&\mp\e\kappa\mathrm{II}_{\partial\Omega} (N,N)\partial_s N -\e\kappa \nabla_{\partial_t \gamma} \mathrm{II}_{\partial\Omega} (N,B)B\mp\e\kappa \mathrm{II}_{\partial\Omega} (N,B)\partial_s B\nonumber\\
	&= \mp\e (\partial_s \kappa) \mathrm{II}_{\partial\Omega} (N,N)N\mp\e(\partial_s\kappa)\mathrm{II}_{\partial\Omega} (N,B)B-\e\kappa\nabla_{\partial_t \gamma} \mathrm{II}_{\partial\Omega} (N,N)N\nonumber\\
	&\pm\e\kappa^2\mathrm{II}_{\partial\Omega} (N,N)T \mp\e\kappa\tau_1\mathrm{II}_{\partial\Omega} (N,N) B -\e\kappa \nabla_{\partial_t \gamma} \mathrm{II}_{\partial\Omega} (N,B)B\nonumber\\
	&\pm\e\kappa\tau_1 \mathrm{II}_{\partial\Omega} (N,B) N.
	\end{align}
Moreover, \eqref{secondT} becomes
	\begin{align}\label{secondTs}
	\partial_\mu\partial_t T=\pm\partial_s\partial_t T&=\pm(\partial_s^2 \kappa) N\pm \partial_s \kappa\partial_s N \pm(\partial_s \kappa) \tau_1 B\pm\kappa (\partial_s \tau_1) B \pm \kappa \tau_1 \partial_s B\nonumber\\
	&=\pm(\partial_s^2 \kappa) N\mp (\partial_s \kappa)\kappa T\pm 2(\partial_s \kappa)\tau_1 B \pm\kappa (\partial_s \tau_1) B \mp \kappa \tau_1^2 N.
	\end{align}
Comparing the coefficients of $B$ in \eqref{firstTs} and \eqref{secondTs}, for $\kappa\neq0$, we obtain that
	\begin{align*}
	\pm \Big(2(\partial_s \kappa)\tau_1&+\kappa (\partial_s \tau_1)\Big)=\mp\e\kappa\tau_1\mathrm{II}_{\partial\Omega} (N,N)-\e\kappa \nabla_{\partial_t \gamma} \mathrm{II}_{\partial\Omega} (N,B)\mp\e(\partial_s\kappa)\mathrm{II}_{\partial\Omega} (N,B)\Longleftrightarrow\\
	 \pm\partial_s \tau_1&=-2\frac{(\partial_s\kappa)}{\kappa}\tau_1\mp\e\tau_1\mathrm{II}_{\partial\Omega} (N,N)-\e\nabla_{\partial_t \gamma} \mathrm{II}_{\partial\Omega} (N,B)\mp\e\frac{(\partial_s\kappa)}{\kappa}\mathrm{II}_{\partial\Omega} (N,B)
	\end{align*}
and from \eqref{sffK}, we have
	\begin{align}\label{boundforderivativeoftau1}
|\partial_s \tau_1|&\le\left|2\frac{(\partial_s\kappa)}{\kappa}\tau_1\right|+|\e\tau_1\mathrm{II}_{\partial\Omega} (N,N)|+|\e\nabla_{\partial_t \gamma} \mathrm{II}_{\partial\Omega} (N,B)|+\left|\e\frac{(\partial_s\kappa)}{\kappa}\mathrm{II}_{\partial\Omega} (N,B)\right|\le C_K.
	\end{align}
Comparing the coefficients of N in \eqref{firstTs} and \eqref{secondTs}, we obtain that
	\begin{align*}
\pm\Big((\partial_s^2 \kappa)+ \kappa \tau_1^2\Big)&=\mp\e (\partial_s \kappa) \mathrm{II}_{\partial\Omega} (N,N)-\e\kappa\nabla_{\partial_t \gamma} \mathrm{II}_{\partial\Omega} (N,N)\pm\e\kappa\tau_1 \mathrm{II}_{\partial\Omega} (N,B)\Longleftrightarrow\\
	&\pm\partial_s^2 \kappa=\mp\kappa\tau_1^2\mp\e (\partial_s \kappa) \mathrm{II}_{\partial\Omega} (N,N)-\e\kappa\nabla_{\partial_t \gamma} \mathrm{II}_{\partial\Omega} (N,N)\pm\e\kappa\tau_1 \mathrm{II}_{\partial\Omega} (N,B)
	\end{align*}
and from \eqref{sffK}, we have
	\begin{align}\label{boundforderivativeofkss}
|\partial_s^2 \kappa|=|\kappa\tau_1^2|+|\e (\partial_s \kappa) \mathrm{II}_{\partial\Omega} (N,N)|+|\e\kappa\nabla_{\partial_t \gamma} \mathrm{II}_{\partial\Omega} (N,N)|+|\e\kappa\tau_1 \mathrm{II}_{\partial\Omega} (N,B)|\le C'_K|\partial_s \kappa|+C''_K |\kappa|.
	\end{align}
Therefore,
	\begin{align}\label{extrater}
	\left|2 \partial_s\kappa \partial_s^2 \kappa\right|&\leq 2 C'_K |\partial_s\kappa|^2+2 C''_K\left|\partial_s\kappa\right||\kappa| \leq 2 C'_K |\partial_s\kappa|^2+C''_K\left(|\partial_s\kappa|^2+\kappa^2\right) \nonumber\\
	&\le 2 C'_K |\partial_s\kappa|^2+C''_K\left(|\partial_s\kappa|^2+1+\kappa^4\right)\nonumber \\
	&\leq C_K\left|\partial_s^2T\right|^2+C_K.
	\end{align}
Using \eqref{boundforderivativeoftau1} and \eqref{extrater}, we have the following calculation.
	\begin{align*}
	|\partial_\mu |Q_2|^2|=|\partial_\mu |\partial_s^2 T|^2|&\le|\partial_\mu (\partial_s \kappa)^2+\partial_\mu (\kappa^4)+\partial_\mu(\kappa^2\tau_1^2)|\\
	&\le|2\partial_s\kappa\partial_s^2\kappa|+|4\kappa^3 \partial_s \kappa|+|2\kappa^2\tau_1^2\partial_s \kappa|+|2\tau_1\kappa^2\partial_s\tau_1|\\
	&\le b_1 |Q_2|^2+b_2,
	\end{align*}
where $b_1$ and $b_2$ depend only on $K$. Applying Theorem \ref{thm_Stahl-localMP} to $Q=Q_m$ with \eqref{eqn_Altschuler-interior} and
\eqref{eqn_Altschuler-boundary} yields a local bound for $Q_m$ on a shrunken cylinder that reaches
$\partial\gamma_t$. Since $Q_m$ controls $|\partial_s^m T|$ with the correct scaling (as in
\cite[Part II, \S 3]{Altschuler91}), we obtain \eqref{eqn_FBCSF-derivative}.
\end{proof}
\section{Monotonicity and Entropy}

In this section we define the entropy of free boundary curve shortening flow. We follow the definitions in Edelen \cite{Edelen2016}.

\begin{definition}\label{def471}
Let $r$ be the cut-off radius, with $r \le r_{\partial\Omega}/c_0(n)$, and set
	\begin{align*}
	\eta(\xi)=(1-\xi)_+^4.
	\end{align*}
Define the cut-off function, at radius $r$, by
	\begin{align*}
	\phi_{\partial\Omega,r}(x,t)=\eta\left(\left(\frac{r^2}{\hat\sigma}\right)^{3/4}\frac{|x|^2-\alpha \hat\sigma}{r^2}\right)
=\left(1-r^{-2}\left(\frac{r^2}{\hat\sigma}\right)^{3/4}\left(|x|^2-\alpha \hat\sigma\right)\right)_+^4,
	\end{align*}
where $\hat\sigma=-t$ and $\alpha=\alpha(n)\ge 1/2$ is to be determined later. Similarly, define the reflected cut-off function by
	\begin{align*}
	\tilde{\phi}_{\partial\Omega,r}(x,t)=\phi_{\partial\Omega,r}(\tilde{x},t), \quad \tilde{x}=2\zeta(x)-x,
	\end{align*}
where $\tilde{x}$ is the reflection of $x$ across $\partial\Omega$, and $\zeta(x)$ is the nearest-point orthogonal projection of $x$ onto $\partial\Omega$.
\end{definition}

\begin{definition}\label{def:reflected_truncated_kernel}
Define the Gaussian heat kernel
	\begin{align*}
	\rho_{\partial\Omega}(x,t)=(4\pi\hat{\sigma})^{-n/2}\exp\left(-\frac{|x|^2}{4\hat{\sigma}}\right),
	\end{align*}
where $\hat{\sigma}=-t>0$, and define the reflected heat kernel by
	\begin{align*}
	\tilde{\rho}_{\partial\Omega}(x,t)=\rho_{\partial\Omega}(\tilde{x},t),
	\end{align*}
with $\tilde{x}=2\zeta(x)-x$ as in Definition \ref{def471}. Define the reflected, truncated heat kernel by
	\begin{align*}
	f_{\partial\Omega,r}(x,t)=\rho_{\partial\Omega}(x,t)\phi_{\partial\Omega,r}(x,t)+\tilde{\rho}_{\partial\Omega}(x,t)\tilde{\phi}_{\partial\Omega,r}(x,t),
	\end{align*}
where $\phi_{\partial\Omega,r},\tilde{\phi}_{\partial\Omega,r}$ are as in Definition \ref{def471}.
\end{definition}

The above definition gives the appropriate Gaussian density centred at the spacetime origin. Given
$X_0=(x_0,t_0)$, define the recentered reflected, truncated heat kernel by
	\begin{align*}
	f_{\partial\Omega,r,X_0}(x,t)
&=\rho_{\partial\Omega}(x-x_0,t-t_0)\,\phi_{\partial\Omega,r}(x-x_0,t-t_0)\\
	&+\rho_{\partial\Omega}(\widetilde{x}-x_0,t-t_0)\,\phi_{\partial\Omega,r}(\widetilde{x}-x_0,t-t_0),
	\end{align*}
where $\widetilde{x}=2\zeta(x)-x$ is the reflection of $x$ across $\partial\Omega$ (with $\zeta(x)$ the nearest-point projection onto $\partial\Omega$).
In particular, $f_{\partial\Omega,r,(0,0)}\equiv f_{\partial\Omega,r}$.

\begin{theorem}[Free boundary monotonicity formula \cite{Edelen2016}]\label{thm_fb_monotonicity}
Let $\{\mu(t)\}_{t\geq -1}$ be a free-boundary Brakke flow supported in $\Omega\subset U$, with $\partial\Omega$.
Assume
	\begin{align*}
	0\in B_{r/10}(\partial\Omega)\cap\overline{\Omega}
\quad\text{ and}\quad
d(0,\partial U)\geq r.
	\end{align*}
Then there exist constants $\sigma_0=\sigma_0(r,n)>0$ and $A=A(r,n)>0$ such that, if $r\leq r_{\partial\Omega}/c_1$, the function
	\begin{align*}
	t\mapsto e^{A(-t)^{1/4}}\int f_{\partial\Omega,r}(\cdot,t)\,d\mu(t)+A\,M(-t)
	\end{align*}
is nonincreasing for $t\in[-\sigma_0,0]$. Here $M$ is any constant satisfying
	\begin{align*}
	M\geq \mu(-\sigma_0)\bigl(\phi_{\partial\Omega,r}(\cdot,-\sigma_0)+\tilde{\phi}_{\partial\Omega,r}(\cdot,-\sigma_0)\bigr).
	\end{align*}
\end{theorem}

We will use Theorem \ref{thm_fb_monotonicity} throughout the next section.

\section{Type I Blowups and Self Shrinkers}\label{section6}

\subsection{Type I Singularities}
Let $\{\gamma_t\}_{t\in[0,t_0)}$ be a smooth free-boundary curve shortening flow in a smooth domain
$\Omega\subset\mathbb{R}^{n+1}$ with boundary $\partial\Omega$, and suppose that a boundary singularity occurs at
$X_0=(x_0,t_0)$ with $x_0\in \partial\Omega$. Assume the singularity is of Type I, i.e.
	\begin{align*}
	|\kappa|(p,t)\le \frac{C}{\sqrt{t_0-t}},
\quad\text{ for all }(p,t)\in\gamma_t\times(t_0-1,t_0).
	\end{align*}

\begin{theorem}[Type I blow-ups are self-shrinkers]\label{thm_typeI_blowups_shrinkers}
Let $\{\gamma_t\}_{t\in[t_0-1,t_0)}$ be a smooth free-boundary curve shortening flow in a smooth domain
$\Omega\subset\mathbb{R}^{n+1}$ with boundary $\partial\Omega$ and suppose $(x_0,t_0)$ is a boundary singular point with $x_0\in \partial\Omega$. Assume the Type I curvature bound
	\begin{align*}
	|\kappa|(p,t)\le \frac{C}{\sqrt{t_0-t}},
\quad\text{ for all }t\in(t_0-1,t_0)
	\end{align*}
and all $p\in\gamma_t$.

Let $\lambda_i\to\infty$ and define the parabolically rescaled flows
	\begin{align*}
	\gamma^{(i)}_t:=\lambda_i\bigl(\gamma_{t_0+t/\lambda_i^2}-x_0\bigr),\qquad t\in(-\infty,0).
	\end{align*}
Let $\Omega_i:=\lambda_i(\Omega-x_0)$ and $\partial\Omega_i=\lambda_i(\partial\Omega-x_0)$. Then $\partial\Omega_i$ converges in
$C^2_{\mathrm{loc}}$ to the tangent hyperplane $L:=T_{x_0}\partial\Omega$ and $\Omega_i$ converges in $C^2_{\mathrm{loc}}$ to the
corresponding half-space $H$ bounded by $L$.

Moreover, after passing to a subsequence, $\gamma^{(i)}_t$ converges smoothly on compact subsets of
$H\times(-\infty,0)$ to an ancient smooth free-boundary curve shortening flow $\{\gamma^\infty_t\}_{t<0}$ in the flat
half-space $H$, with free-boundary on $L$. Any such blow-up limit is self-similar, and each time-slice satisfies the
self-shrinker equation in $H$,
	\begin{align*}
	\vec{H}+\frac{x^\perp}{2\hat\sigma}=0,\qquad \hat\sigma=-t,
	\end{align*}
together with the free-boundary condition on $L$.
\end{theorem}

\begin{proposition}[Compactness of blow-ups {\cite{Edelen2016}}]\label{prop:edelen_blowup_compactness}
Let $X_0=(x_0,t_0)$ with $x_0\in\overline{\Omega}$ and $t_0>-1$, and let $\lambda_i\to 0$.
Then, after passing to a subsequence, there exists an ancient Brakke flow $\mathcal{M}'$ such that
	\begin{align*}
	\mathcal{D}_{1/\lambda_i}(\mathcal{M}-X_0)\to \mathcal{M}'
	\end{align*}
as Brakke flows. Moreover, $\mathcal{M}'$ is either
\begin{itemize}
\item a Brakke flow in $\mathbb{R}^N$ if $x_0\notin \partial\Omega$, or
\item a free-boundary Brakke flow in a half-space in $\mathbb{R}^N$ if $x_0\in \partial\Omega$.
\end{itemize}
If $\mathcal{M}'$ has free boundary, then reflecting across the limiting boundary hyperplane yields a Brakke flow
$\widetilde{\mathcal{M}}'$ in $\mathbb{R}^N$ without boundary. Otherwise set $\widetilde{\mathcal{M}}'=\mathcal{M}'$.
\end{proposition}

\begin{lemma}[cf. \cite{Edelen2016}](Reflected self-shrinker from Type I blow-up).\label{lemma1}
Let $\left\{\gamma_t\right\}$ be a free-boundary curve shortening flow in a smooth domain with boundary, satisfying a Type I bound near a boundary point $\left(x_0, t_0\right)$. Apply Edelen's reflected monotonicity formula with the recentered reflected heat kernel and perform a parabolic blow-up at $\left(x_0, t_0\right)$. Then any blow-up limit $\left\{\gamma_t\right\}_{t<0}$ is a free-boundary self-shrinker in a half-plane, i.e. it satisfies
	\begin{align*}
\kappa N+\frac{x^\perp}{2\hat\sigma}=0,\qquad \hat\sigma=-t,
	\end{align*}
on each time-slice $\gamma_t$, with free-boundary condition on the limiting boundary line.
Moreover, reflecting $\gamma_t$ across this boundary line produces a smooth interior self-shrinker $\Sigma\subset \mathbb{R}^2$.
\end{lemma}

\begin{proof}
Pick $t_j \to t_0$, set $\lambda_j:=\frac{1}{\sqrt{t_0-t_j}}\to\infty$, and define the rescaled flows
	\begin{align*}
\gamma_t^{(j)}:=\lambda_j \bigl(\gamma_{t_0+t / \lambda_j^2}-x_0\bigr), \quad t \in\left(- T, 0\right).
	\end{align*}
Set $\sigma:=t_0-t$ and $\hat\sigma:=-t$ (so along the rescaled flow, we have $\sigma=(-t)/\lambda_j^2$ and $\hat\sigma=\lambda_j^2\sigma$). From now on, $(x, t)$ will denote the rescaled space-time variables and the corresponding original variables will be $\left(x_0+x / \lambda_j, t_0+t / \lambda_j^2\right)$. Type I is scale-invariant, so
	\begin{align*}
\left|\kappa^{(j)}(x, t)\right| \leq \lambda_j^{-1} \frac{C}{\sqrt{\sigma}}=\frac{C}{\sqrt{\hat{\sigma}}}.
	\end{align*}
Let $\Omega_j:=\lambda_j\left(\Omega-x_0\right)$ and $\partial \Omega_j=\lambda_j\left(\partial \Omega-x_0\right)$. Since $x_0 \in \partial \Omega$ and $\partial \Omega$ is $C^2$, the rescaled boundaries $\partial\Omega_j$ converge in $C_{\mathrm{loc}}^2$ to the tangent line $L:=T_{x_0}(\partial \Omega)$, and the rescaled domains $\Omega_j$ converge in $C_{\mathrm{loc}}^2$ to the corresponding half-plane $H$ bounded by $L$. Using the curvature and derivative estimates up to the boundary (so that on each compact subset of $H \times(-\infty, 0)$ we have uniform $C^k$ bounds), Arzel\`a--Ascoli implies that, after passing to a subsequence,
	\begin{align*}
\gamma_t^{(j)} \longrightarrow \gamma_{\infty} \text{ smoothly on compact subsets of } H \times(-\infty, 0),
	\end{align*}
where $\gamma_{\infty}$ is an ancient free-boundary CSF in the flat half-plane $H$. The cylinder we work on is in the rescaled coordinates $(x, t) \in B_R(0) \times[-T, 0)$. In original variables, this corresponds to
	\begin{align*}
\left(x_0+x / \lambda_j, t_0+t / \lambda_j^2\right) \in B_{R / \lambda_j}(x_0) \times\left[t_0- T/ \lambda_j^2, t_0\right),
	\end{align*}
which is a neighbourhood of $(x_0,t_0)$ for large $j$ and is contained in the region where Edelen's monotonicity holds (with kernel re-centred at $(x_0,t_0)$).

By translating space and time, we may assume $x_0=0$ and $t_0=0$. From Definition \ref{def471}, $\eta(\xi)=1$ whenever $\xi \leq 0$, $\eta(\xi) \in[0,1]$, and $\eta$ decreases from $1$ to $0$ as $\xi$ goes from $0$ to $1$. Let
	\begin{align*}
\psi(x, \sigma):=\left(\frac{r^2}{\sigma}\right)^{3 / 4} \frac{|x|^2-\alpha \sigma}{r^2}.
	\end{align*}
In the blow-up around $(0,0)$, we rescale $x\to\lambda_j^{-1} x$ and $\sigma=\frac{-t}{\lambda_j^2}$. Plugging this into $\psi$ gives
	\begin{align*}
\psi_j(x, t)&=\left(\frac{r^2}{(-t)/\lambda_j^2}\right)^{3 / 4} \frac{\lambda_j^{-2}|x|^2-\alpha (-t)/\lambda_j^2}{r^2}
=r^{-1 / 2} \lambda_j^{-1 / 2} (-t)^{-3 / 4}\left(|x|^2-\alpha(-t)\right).
	\end{align*}
Fix $R$ and $t\in[-T,0)$. Then $|x|\le R$ implies $|x|^2-\alpha(-t)$ is bounded, and $\lambda_j^{-1/2}\to 0$, hence $\psi_j(x,t)\to 0$ uniformly on compact sets. Therefore
	\begin{align*}
\phi^{(j)}(x, t):=\phi_{\partial\Omega, r}\left(\lambda_j^{-1} x, \lambda_j^{-2} t\right)=\eta\left(\psi_j(x, t)\right) \longrightarrow \eta(0)=1,
	\end{align*}
uniformly on compact sets in the rescaled variables, and similarly $\tilde{\phi}^{(j)}\to 1$. In particular,
	\begin{align*}
\phi^{(j)} \rightarrow 1, \quad \tilde{\phi}^{(j)} \rightarrow 1, \quad \nabla \phi^{(j)}\rightarrow 0, \quad \partial_t \phi^{(j)}\rightarrow 0.
	\end{align*}
Also $\partial\Omega_j \rightarrow L$ implies $\tilde{x}$ is reflection across $L$ in the limit, and
	\begin{align*}
e^{A\left((-t)/\lambda_j^2 \right)^{1 / 4}} \longrightarrow 1, \qquad A M\left((-t)/\lambda_j^2\right) \longrightarrow 0.
	\end{align*}

Applying Edelen's reflected monotonicity formula with the recentered reflected heat kernel to the original flow near $(x_0,t_0)$, rewriting it for the rescaled flows $\gamma^{(j)}$, and using the above convergences, we may pass to the limit along the subsequence $\gamma^{(j)}\to\gamma_\infty$ to obtain the exact reflected Huisken monotonicity identity in the flat half-plane $H$:
	\begin{align}\label{finaleq}
\frac{d}{d t} \int_{\gamma_{\infty,t}}(\rho+\tilde{\rho}) \, d\mu_t
=-\int_{\gamma_{\infty,t}}\left|\kappa N+\frac{x^{\perp}}{2(-t)}\right|^2(\rho+\tilde{\rho}) \, d\mu_t,
	\end{align}
for $t<0$, where $\rho(x,t)=(4\pi(-t))^{-1/2}\exp\left(-\frac{|x|^2}{4(-t)}\right)$ and $\tilde\rho(x,t)=\rho(\tilde x,t)$ with $\tilde x$ the reflection across $L$.

Next, let
	\begin{align*}
\Phi(t):=\int_{\gamma_t}\left(\rho_{(x_0,t_0)}+\tilde\rho_{(x_0,t_0)}\right)\,d\mu_t
	\end{align*}
denote the reflected Gaussian quantity for the original flow, centred at $(x_0,t_0)$. By Edelen's monotonicity formula, $\Phi(t)$ is monotone nonincreasing for $t$ sufficiently close to $t_0$, hence the limit
	\begin{align*}
\Theta:=\lim_{t\uparrow t_0}\Phi(t)
	\end{align*}
exists. Fix any $t<0$. For each $j$, by the parabolic change of variables defining $\gamma_t^{(j)}$ and the corresponding scaling of the heat kernels, the quantity
	\begin{align*}
\int_{\gamma_t^{(j)}}(\rho+\tilde\rho)\,d\mu_t
	\end{align*}
is exactly $\Phi\left(t_0+t/\lambda_j^2\right)$ (with the same centre $(x_0,t_0)$). Since $t_0+t/\lambda_j^2 \to t_0$ as $j\to\infty$, we obtain
	\begin{align*}
\lim_{j\to\infty}\int_{\gamma_t^{(j)}}(\rho+\tilde\rho)\,d\mu_t
=\lim_{j\to\infty}\Phi\left(t_0+t/\lambda_j^2\right)=\Theta.
	\end{align*}
Passing to the smooth limit $\gamma_t^{(j)}\to\gamma_{\infty,t}$ on compact sets and using dominated convergence for the Gaussian weight yields
	\begin{align*}
\int_{\gamma_{\infty,t}}(\rho+\tilde\rho)\,d\mu_t=\Theta
	\end{align*}
for every $t<0$. Thus the left-hand side of \eqref{finaleq} vanishes for all $t<0$, and hence
	\begin{align*}
\kappa N+\frac{x^\perp}{2(-t)}\equiv 0 \quad \text{ on } \gamma_{\infty,t}, \qquad t<0.
	\end{align*}
Therefore $\gamma_{\infty,t}$ is a free-boundary self-shrinker in the half-plane $H$.

Finally, since the limiting boundary is a line and the free-boundary condition is orthogonality, reflecting $\gamma_{\infty,t}$ across $L$ produces a smooth interior self-shrinker $\Sigma\subset\mathbb{R}^2$.
\end{proof}

\begin{lemma}[Uniform reflected Gaussian bound and scaling]\label{lem_uniform-reflected-gaussian-bound}
Let $\{\gamma_t\}_{t\in[0,t_0)}$ be a smooth free-boundary curve shortening flow in a smooth domain $\Omega\subset\mathbb{R}^{n+1}$ with orthogonal boundary condition.
Fix a boundary spacetime point $X_0=(x_0,t_0)$ and a truncation scale $r>0$.
Let $f_{\partial\Omega,r,X_0}(x,t)$ be the truncated reflected kernel from Definition \ref{def:reflected_truncated_kernel} and set
	\begin{align*}
	\Phi_{r,X_0}(t):=\int_{\gamma_t} f_{\partial\Omega,r,X_0}(\cdot,t)\, d\mathcal{H}^1,\quad t<t_0.
	\end{align*}
Assume the reflected entropy is finite and dominates these functionals, namely
	\begin{align*}
	\operatorname{Ent}_{\partial\Omega}[\{\gamma_t\}]
:=
\sup_{(r,X)}\ \sup_{s<t_X}\ \int_{\gamma_s} f_{\partial\Omega,r,X}(\cdot,s)\, d\mathcal{H}^1
<\infty,
	\end{align*}
where $X=(x_X,t_X)$ ranges over boundary spacetime points and $r>0$.
Then for every $t<t_0$,
	\begin{align*}
	\Phi_{r,X_0}(t)\le \operatorname{Ent}_{\partial\Omega}[\{\gamma_t\}].
	\end{align*}

Now choose $t_j\to t_0$, set $\lambda_j:=(t_0-t_j)^{-1/2}$, and define the parabolic rescalings
	\begin{align*}
	\gamma_t^{(j)}:=\lambda_j(\gamma_{t_0+t/\lambda_j^2}-x_0),\quad t\in(-\lambda_j^2 t_0,0).
	\end{align*}
Let
	\begin{align*}
	\rho(x,t):=(4\pi(-t))^{-1/2}\exp\Bigl(-\frac{|x|^2}{4(-t)}\Bigr)
	\end{align*}
and let $\tilde\rho(x,t):=\rho(\tilde x,t)$ be the reflected Gaussian after flattening the boundary at $x_0$.
Let $\phi^{(j)}$, $\tilde\phi^{(j)}$ be the rescaled cut-offs induced from Definition \ref{def:reflected_truncated_kernel} under $(x,t)\mapsto(x_0+x/\lambda_j,t_0+t/\lambda_j^2)$.
Define
	\begin{align*}
	\Phi^{(j)}(t):=\int_{\gamma_t^{(j)}}\Bigl(\rho(\cdot,t)\phi^{(j)}(\cdot,t)+\tilde\rho(\cdot,t)\tilde\phi^{(j)}(\cdot,t)\Bigr)\, d\mathcal{H}^1.
	\end{align*}
Then for each fixed $t<0$,
	\begin{align*}
	\Phi^{(j)}(t)=\Phi_{r,X_0}\bigl(t_0+t/\lambda_j^2\bigr)\le \operatorname{Ent}_{\partial\Omega}[\{\gamma_t\}].
	\end{align*}
In particular, for every $T<\infty$,
	\begin{align*}
	\sup_j\ \sup_{t\in[-T,0)}\Phi^{(j)}(t)\le \operatorname{Ent}_{\partial\Omega}[\{\gamma_t\}.
	\end{align*}.
\end{lemma}

\begin{proof}
The first inequality is immediate from the defining supremum for $\operatorname{Ent}_{\partial\Omega}[\{\gamma_t\}]$.

Fix $t<0$ and write $s=t_0+t/\lambda_j^2$.
The rescaling map is $x'= \lambda_j(x-x_0)$, so along curves $d\mathcal{H}^1_{x'}=\lambda_j d\mathcal{H}^1_x$.
A direct computation gives
	\begin{align*}
	\rho(x',t)=\lambda_j^{-1}\rho(x-x_0,t/\lambda_j^2).
	\end{align*}
Hence $\rho(\cdot,t)\, d\mathcal{H}^1$ is invariant under the parabolic rescaling.
Since $\phi^{(j)}$, $\tilde\phi^{(j)}$ are pullbacks of the original cut-offs, the same change of variables yields
	\begin{align*}
	\int_{\gamma_t^{(j)}}\rho(\cdot,t)\phi^{(j)}(\cdot,t)\, d\mathcal{H}^1
=
\int_{\gamma_s}\rho_{x_0,t_0}(\cdot,s)\phi_{\partial\Omega,r,x_0}(\cdot,s)\, d\mathcal{H}^1,
	\end{align*}
and likewise for the reflected term.
By Definition \ref{def:reflected_truncated_kernel}, the right-hand side equals $\Phi_{r,X_0}(s)$.
This proves $\Phi^{(j)}(t)=\Phi_{r,X_0}(t_0+t/\lambda_j^2)$, and the bound follows from the first part.
\end{proof}

\begin{lemma}[Limit passage for reflected monotonicity via fixed cut-offs]\label{lem_limit-passage-fixed-cutoff}
Assume the setting of Lemma \ref{lem_uniform-reflected-gaussian-bound}.
Assume moreover that, after passing to a subsequence, the rescaled flows $\gamma_t^{(j)}$ converge smoothly on compact subsets of $\mathbb{R}^{n+1}\times(-\infty,0)$ to a smooth limit flow $\gamma_t^{(\infty)}$ (after flattening and reflection as in Section \ref{section6}).

Let $\eta_R\in C_c^\infty(\mathbb{R}^{n+1})$ satisfy $0\le\eta_R\le 1$, $\eta_R\equiv 1$ on $B_R(0)$, $\eta_R\equiv 0$ on $\mathbb{R}^{n+1}\setminus B_{2R}(0)$, and $|\nabla\eta_R|+|\nabla^2\eta_R|\le C/R$. Then for every $t_1<t_2<0$ and every $R<\infty$,
	\begin{align*}
	\int_{\gamma_{t_2}^{(\infty)}}(\rho+\tilde\rho)\eta_R\, d\mathcal{H}^1
-
\int_{\gamma_{t_1}^{(\infty)}}(\rho+\tilde\rho)\eta_R\, d\mathcal{H}^1
&=
-\int_{t_1}^{t_2}\int_{\gamma_t^{(\infty)}}
\Bigl|\kappa^{(\infty)}+\frac{\langle x,\nu^{(\infty)}\rangle}{2(-t)}\Bigr|^2
(\rho+\tilde\rho)\eta_R\, d\mathcal{H}^1 dt\\
&+ \operatorname{Ent}_\infty(R;t_1,t_2),
	\end{align*}
where the error term satisfies
	\begin{align*}
	|\operatorname{Ent}_\infty(R;t_1,t_2)|
\le \frac{C}{R}\int_{t_1}^{t_2}\int_{\gamma_t^{(\infty)}\cap(B_{2R}(0)\setminus B_R(0))}
(\rho+\tilde\rho)\, d\mathcal{H}^1 dt.
	\end{align*}
In particular, since the limit has finite reflected Gaussian integral for each $t<0$ by Lemma \ref{lem_uniform-reflected-gaussian-bound} and Fatou,
letting $R\to\infty$ yields the exact reflected Huisken identity for the limit
	\begin{align*}
	\int_{\gamma_{t_2}^{(\infty)}}(\rho+\tilde\rho)\, d\mathcal{H}^1
-
\int_{\gamma_{t_1}^{(\infty)}}(\rho+\tilde\rho)\, d\mathcal{H}^1
=
-\int_{t_1}^{t_2}\int_{\gamma_t^{(\infty)}}
\Bigl|\kappa^{(\infty)}+\frac{\langle x,\nu^{(\infty)}\rangle}{2(-t)}\Bigr|^2
(\rho+\tilde\rho)\, d\mathcal{H}^1 dt.
	\end{align*}
\end{lemma}

\begin{proof}
Since $\gamma_t^{(\infty)}$ is smooth and (after reflection across the flat boundary line) is an interior smooth curve shortening flow, the standard computation for the quantity
	\begin{align*}
	\int_{\gamma_t^{(\infty)}}(\rho+\tilde\rho)\eta_R\, d\mathcal{H}^1
	\end{align*}
gives the stated identity, with an error supported where $\nabla\eta_R\ne 0$, namely in $B_{2R}(0)\setminus B_R(0)$.
The bound for $\operatorname{Ent}_\infty(R;t_1,t_2)$ follows from $|\nabla\eta_R|+|\nabla^2\eta_R|\le C/R$.

By Lemma \ref{lem_uniform-reflected-gaussian-bound} and Fatou, for each fixed $t<0$,
	\begin{align*}
	\int_{\gamma_t^{(\infty)}}(\rho+\tilde\rho)\, d\mathcal{H}^1<\infty.
	\end{align*}
Hence the right-hand side bound implies $\operatorname{Ent}_\infty(R;t_1,t_2)\to 0$ as $R\to\infty$.
Also $\eta_R\to 1$ pointwise, so monotone convergence yields convergence of the cut-off integrals to the full integrals, giving the final identity.
\end{proof}

\begin{lemma}[No flat tangent flow at a singular point]\label{lemma3}
Let $\{\gamma_t\}_{t<t_0}$ be a smooth curve shortening flow in $\mathbb R^{n+1}$, and let $(x_0,t_0)$ be a singular spacetime point of the flow. Then no tangent flow at $(x_0,t_0)$ is a multiplicity-one static straight line.
\end{lemma}

\begin{proof}
Fix a tangent flow at $(x_0,t_0)$ obtained by parabolic rescalings
	\begin{align*}
	\gamma^{(j)}_t:=\lambda_j\bigl(\gamma_{t_0+t/\lambda_j^2}-x_0\bigr),\qquad \lambda_j\to\infty,
	\end{align*}
and pass to a subsequence so that $\gamma^{(j)}\to\gamma^\infty$ smoothly on compact subsets of $\mathbb R^{n+1}\times(-\infty,0)$ (as in the Type I compactness argument).

Assume for contradiction that $\gamma^\infty_t\equiv L$ is a multiplicity-one straight line (independent of $t$). Then for every fixed $t<0$ and every $R<\infty$,
	\begin{align*}
	\int_{\gamma^{(j)}_t\cap B_R(0)} \rho(\cdot,t)\,d\mu^{(j)}_t
\longrightarrow
\int_{L\cap B_R(0)} \rho(\cdot,t)\,d\mathcal H^1,
	\end{align*}
and letting $R\to\infty$ gives
	\begin{align*}
	\lim_{j\to\infty}\int_{\gamma^{(j)}_t} \rho(\cdot,t)\,d\mu^{(j)}_t= \int_L \rho(\cdot,t)\,d\mathcal H^1
=1,
	\end{align*}
since the Gaussian density of a multiplicity-one line equals $1$.

Translating back to the unrescaled flow, this identity is exactly the statement that the Gaussian density of $\gamma_t$ at $(x_0,t_0)$ equals $1$:
	\begin{align*}
	\Theta(\gamma,(x_0,t_0))=\lim_{t\uparrow t_0}\int_{\gamma_t}\rho_{(x_0,t_0)}(\cdot,t)\,d\mu_t=1.
	\end{align*}
By White's local regularity theorem there exists $\e>0$ such that if $\Theta(\gamma(x_0,t_0))\le 1+\e$, then $(x_0,t_0)$ is a regular point of the flow. In particular, density $1$ forces regularity at $(x_0,t_0)$, contradicting that $(x_0,t_0)$ is singular. Hence $\gamma^\infty$ cannot be a multiplicity-one line.
\end{proof}

\begin{lemma}[\cite{Edelen2016}](Entropy passes to tangent flows).\label{lem_entropy_to_tangent}
Let $\{\gamma_t\}_{t\in[0,t_0)}$ be a smooth curve shortening flow with $\operatorname{Ent}(\gamma_0)<\infty$.
Let $\Sigma$ be any time-slice of any tangent flow at $(x_0,t_0)$. Then
	\begin{align*}
	\operatorname{Ent}(\Sigma)\le \operatorname{Ent}(\gamma_0).
	\end{align*}
\end{lemma}

\begin{proof}
For each $(y,\theta)$ with $\theta>t$, the Gaussian functional
	\begin{align*}
	\mathcal F_{(y,\theta)}(\gamma_t):=\int_{\gamma_t}\rho_{(y,\theta)}(\cdot,t)\,d\mu_t
	\end{align*}
is nonincreasing in $t$ (Huisken monotonicity in the interior case, and Edelen's reflected monotonicity in the free-boundary case). In particular,
	\begin{align*}
	\mathcal F_{(y,\theta)}(\gamma_t)\le \mathcal F_{(y,\theta)}(\gamma_0)\le \operatorname{Ent}(\gamma_0)
\quad\text{ for all }t<\theta.
	\end{align*}
Let $\gamma^{(j)}$ be any blow-up sequence at $(x_0,t_0)$ converging to a tangent flow $\gamma^\infty$.
Fix $(y,\theta)$ in the rescaled variables and a fixed time $t<\theta$.
By smooth (hence measure) convergence on compact subsets and Gaussian decay, we have
	\begin{align*}
	\mathcal F_{(y,\theta)}(\gamma^\infty_t)
=\lim_{j\to\infty}\mathcal F_{(y,\theta)}(\gamma^{(j)}_t).
	\end{align*}
Undoing the scaling identifies $\mathcal F_{(y,\theta)}(\gamma^{(j)}_t)$ with a Gaussian functional of the original flow at a time $t_0+t/\lambda_j^2\to t_0$, hence each term is bounded by $\operatorname{Ent}(\gamma_0)$ by the previous inequality. Therefore
	\begin{align*}
	\mathcal F_{(y,\theta)}(\gamma^\infty_t)\le \operatorname{Ent}(\gamma_0)
\quad\text{ for all }(y,\theta),\ t<\theta.
	\end{align*}
Taking the supremum over $(y,\theta)$ in the definition of entropy gives $\operatorname{Ent}(\Sigma)\le \operatorname{Ent}(\gamma_0)$ for any time-slice $\Sigma$ of the tangent flow.
\end{proof}

\begin{theorem}\label{thm_tau1_planar}
If $\tau_1\equiv 0$ on an interval, then $\gamma$ is planar on that interval. In particular, if $\tau_1\equiv 0$ along $\gamma$ then $\gamma$ is contained in an affine $2$-plane.
\end{theorem}
\begin{proof}
Let $\gamma:I\to\mathbb R^{n+1}$ be a unit-speed $C^3$ curve and fix an open interval $I_0\subset I$ on which $\kappa>0$, so that the Frenet frame $(T,N,B_1,\dots,B_{n-2})$ is defined on $I_0$.
Assume $\tau_1\equiv 0$ on $I_0$. Then the first two Frenet equations reduce to
	\begin{align}\label{eq_tau1_frenet}
	\partial_s T=\kappa N,\qquad \partial_s N=-\kappa T.
	\end{align}
Consider the $2$-vector field
	\begin{align*}
	\mathrm{II}(s):=T(s)\wedge N(s)\in \Lambda^2(\mathbb R^{n+1}),\qquad s\in I_0.
	\end{align*}
Differentiating and using \eqref{eq_tau1_frenet} gives
	\begin{align*}
	\partial_s \mathrm{II}
&=(\partial_s T)\wedge N+T\wedge(\partial_s N)
=\kappa N\wedge N+T\wedge(-\kappa T)=0.
	\end{align*}
Hence $\mathrm{II}(s)$ is constant on $I_0$, and therefore the oriented $2$-plane
	\begin{align*}
	P:=\operatorname{span}\{T(s),N(s)\}
	\end{align*}
is independent of $s\in I_0$. Since $\partial_s \gamma=T$, we have $T(s)\in P$ for all $s\in I_0$, so $\gamma(I_0)$ is contained in the affine plane $\gamma(s_*)+P$ for any fixed $s_*\in I_0$.

If $\kappa$ vanishes somewhere, the curve is locally a straight line there, hence still contained in an affine $2$-plane; covering $I$ by intervals where either $\kappa>0$ or $\kappa\equiv 0$ yields the claim.
Finally, once $\gamma$ is planar, all higher torsions $\tau_i$ vanish identically.
\end{proof}

By the classification of self-shrinking curves for CSF due to Abresch and Langer \cite{AL86}, any closed planar self-shrinker is either a round circle or an Abresch-Langer curve. By Theorem \ref{thm_FBCSF-classification}, we have the following.
\begin{lemma}[Low-entropy planar shrinkers]\label{lemma4}
Let $\Sigma \subset \mathbb{R}^{n+1}$ be a smooth self-shrinking solution of curve shortening flow, i.e.
	\begin{align*}
\kappa N+\frac{x^{\perp}}{2}=0.
	\end{align*}
Assume $\operatorname{Ent}(\Sigma)<2$. Then either
\begin{itemize}
\item $\Sigma$ is a multiplicity-one straight line, or
\item $\Sigma$ is a multiplicity-one round circle (of radius $\sqrt{2}$), lying in some affine $2$-plane.
\end{itemize}
In particular, if $\Sigma$ arises as a tangent flow at a singular spacetime point, from Lemma \ref{lemma3} option (1) is impossible, so $\Sigma$ must be the round circle.
\end{lemma}

After flattening the boundary and performing the parabolic blow-up at $(x_0,t_0)$, Lemma \ref{lemma1} yields a free-boundary self-shrinker in a half-plane. Reflecting across the limiting boundary line produces an interior self-shrinker $\Sigma\subset\mathbb R^{n+1}$.

By entropy inheritance under blow-up (Lemma \ref{lem_entropy_to_tangent}), the reflected shrinker satisfies
	\begin{align*}
	\operatorname{Ent}(\Sigma)\le \operatorname{Ent}(\gamma_0)<2.
	\end{align*}
By Lemma \ref{lemma3}, $\Sigma$ is not a multiplicity-one line. In particular, $\Sigma$ has no non-compact ends. Indeed, any complete non-compact self-shrinking curve has a blow-down at spatial infinity which is a multiplicity-one line, and this forces $\operatorname{Ent}(\Sigma)\ge 2$, contradicting $\operatorname{Ent}(\Sigma)<2$. Hence $\Sigma$ is compact and from Theorem \ref{thm_tau1_planar}, $\Sigma$ is a closed planar self-shrinker.

By Lemma \ref{lemma4} (low-entropy planar shrinkers), the only non-flat closed self-shrinker with $\operatorname{Ent}<2$ is the round circle. Thus $\Sigma$, is a round circle. Reflecting back across the boundary line, the original free-boundary tangent flow in the half-plane is the corresponding half-circle meeting the boundary line orthogonally.

\section{Type II Blowups and Translators}

\subsection{Planarity}
\begin{theorem}[cf. \cite{Altschuler91}]For a solution $\gamma$ to the curve shortening flow with free boundary, we have
	\begin{align*}
\frac{d}{d t} \int_\gamma|\kappa| d s \leq \left.\frac{\partial |\kappa|}{\partial s}\right|_{\partial \gamma}-\int_\gamma \tau^2|\kappa| d s.
\end{align*}
\end{theorem}
\begin{proof}
From
	\begin{align*}
\frac{\partial \kappa}{\partial t}=\frac{\partial^2 \kappa}{\partial s^2}+\kappa^3-\tau^2 \kappa,
	\end{align*}
we may derive
	\begin{align*}
\frac{\partial}{\partial t}\left(\kappa^2\right)=\frac{\partial^2}{\partial s^2}\left(\kappa^2\right)-2\left(\frac{\partial \kappa}{\partial s}\right)^2+2 \kappa^4-2 \tau^2 \kappa^2.
	\end{align*}
A technical difficulty of this theorem is that for space curves $\kappa^2=\left|\frac{\partial T}{\partial s}\right|^2$ is a more natural function to study than $\kappa$. In the special case of a planar curve, by choosing a consistent normal field one may define an ``inside`` and an ``outside``. Then $\kappa>0$ or $\kappa<0$ makes sense.
Hence, following a suggestion of R. Hamilton, we will make use of the function $\sqrt{\kappa^2+\e}$. For simplicity, denote $K_\e=\sqrt{\kappa^2+\e}$ where $\e>0$. The derived equation for this quantity is
	\begin{align*}
\frac{\partial K_\e}{\partial t}=\frac{\partial^2 K_\e}{\partial s^2}+\frac{1}{K_\e^3} \kappa^2\left(\frac{\partial \kappa}{\partial s}\right)^2+\frac{1}{K_\e}\left(-\left(\frac{\partial \kappa}{\partial s}\right)^2+\kappa^4-\tau^2 \kappa^2\right).
	\end{align*}
Since $\kappa<K_\e$ for all $\e>0$, we have
	\begin{align*}
\frac{d}{d t} \int_\gamma K_\e d s&= \int_\gamma\left[\frac{\partial^2 K_\e}{\partial s^2}+\frac{1}{K_\e^3} \kappa^2\left(\frac{\partial \kappa}{\partial s}\right)^2\right] d s +\int_\gamma\left[\frac{1}{K_\e}\left(-\left(\frac{\partial \kappa}{\partial s}\right)^2+\kappa^4-\tau^2 \kappa^2\right)-K_\e \kappa^2\right] d s \\
	&=\left.\frac{\partial K_{\e}}{\partial s}\right|_{\partial \gamma} + \int_\gamma\left[\frac{1}{K_\e^3} \kappa^2\left(\frac{\partial \kappa}{\partial s}\right)^2+\frac{1}{K_\e}\left(-\left(\frac{\partial \kappa}{\partial s}\right)^2+\kappa^4-\tau^2 \kappa^2\right)-K_\e \kappa^2\right] d s \\
	&\leq \left.\frac{\partial K_{\e}}{\partial s}\right|_{\partial \gamma} -\int_\gamma \frac{1}{K_\e} \tau^2 \kappa^2 d s,
	\end{align*}
where
	\begin{align*}
\frac{\partial K_{\e}}{\partial s}\Big|_{\partial \gamma}:=\frac{\partial K_{\e}}{\partial s}(\text{ endpoint } b)-\frac{\partial K_{\e}}{\partial s}(\text{ endpoint } a).
	\end{align*}
The result follows from letting $\e \rightarrow 0$.

\end{proof}

\begin{theorem}[cf. \cite{Altschuler91}]\label{dissk}
Let $\left\{\gamma_t\right\}_{t \in[0, \omega)}$ be a smooth curve shortening flow in $\mathbb{R}^{n+1}$, with free boundary $\partial\Omega$, meeting $\partial\Omega$ orthogonally for all $t$. Let $\left\{\left(p_n, t_n\right)\right\}$ be an essential blow-up sequence. Then there exist constants $d_0, d_1, d_2>0$, depending only on $\rho$ (and the fixed geometry of the support curve on the boundary), such that the following hold.
	\begin{enumerate}
\item The temporal loss of $\kappa\left(p_n, \cdot\right)$ is bounded from below
	\begin{align*}
\left|\kappa\left(p_n, t\right)\right| \geq \frac{1}{\sqrt{2}}\left|\kappa\left(p_n, t_n\right)\right|, \quad \text{ for} \ t \in\left[t_n, t_n+\frac{d_1}{M_{t_n}}\right].
	\end{align*}
\item The spatial loss of $\kappa(\cdot, t)$ is bounded from below
	\begin{align*}
|\kappa(p, t)| \geq \frac{1}{\sqrt{2}}\left|\kappa\left(p_n, t\right)\right|, \quad \text{ for} \ \operatorname{dist}_t\left\{p, p_n\right\} \leq \sqrt{\frac{d_2}{M_{t_n}}}, \ \ t \in\left[t_n, t_n+\frac{d_1}{M_{t_n}}\right].
	\end{align*}
\item Hence,
	\begin{align*}
|\kappa(p, t)| \geq \frac{1}{2}\left|\kappa\left(p_n, t_n\right)\right|, \quad \text{ for} \ (p, t) \in \mathbf{N}\left(p_n, t_n, d_0\right),
	\end{align*}
where for $d \in \mathbb{R}_+$ and $\left(p_n, t_n\right) \in S^1 \times[0, \omega)$,
	\begin{align*}
	\mathbf{N}\left(p_n, t_n, d\right)=\left\{(p, t) \in S^1 \times\right.&{\left[t_n, \omega\right) \mid \operatorname{dist}_{t_n}\left\{p_n, p\right\} } \left.\leq \sqrt{\frac{d}{M_{t_n}}},\left|t_n-t\right| \leq \frac{d}{M_{t_n}}\right\}.
	\end{align*}
	\end{enumerate}
\end{theorem}
\begin{proof}
Recall that \eqref{eqn_FBCSF-derivative} holds up to the boundary, so we do not need to separate the interior and boundary cases. The proof follows from \cite{Altschuler91}.
\end{proof}

\begin{theorem}[Boundary essential blow-up neighbourhood]
Let $\Omega \subset \mathbb{R}^{n+1}$ be a smooth domain with boundary $\partial \Omega$, and assume $\left\|\mathrm{II}_{\partial\Omega}\right\|_{L^{\infty}} \leq K$. Let $\left\{\gamma_t\right\}_{t \in[0, \omega)}$ be a smooth curve-shortening flow in $\bar{\Omega}$. Fix a non-collapsing parameter $\rho>0$ and let $\left\{\left(p_n, t_n\right)\right\}$ be an essential blow-up sequence with
	\begin{align}\label{essentialseq}
p_n \in \partial \gamma_{t_n}, \quad M_{t_n}:=\max_{\gamma_{t_n}}\left|\kappa\left(\cdot, t_n\right)\right|=\left|\kappa\left(p_n, t_n\right)\right| \rightarrow \infty.
	\end{align}
Then there exist constants $d_0, d_1, d_2>0$, depending only on $\rho$ and $K$, such that
	\begin{align*}
|\kappa(p, t)| \geq \frac{1}{2}\left|\kappa\left(p_n, t_n\right)\right|,
	\end{align*}
whenever
	\begin{align*}
\left|t-t_n\right| \leq \frac{d_0}{M_{t_n}}, \quad \operatorname{dist}_{t_n}\left(p, p_n\right) \leq \sqrt{\frac{d_0}{M_{t_n}}}
	\end{align*}
and $p$ lies on the connected component of $\gamma_{t_n}$ starting at the endpoint $p_n$.
\end{theorem}

\begin{theorem}[\cite{Altschuler91}]\label{thm_disstorsion}
Let $\{\gamma_t\}_{t\in[0,\omega)}$ be a smooth curve shortening flow in $\mathbb R^{n+1}$ with free boundary $\partial\Omega$, meeting $\partial\Omega$ orthogonally for all $t$. Let $\{(p_n,t_n)\}$ be an essential blow-up sequence and assume that there exists $\mu>0$ such that
	\begin{align*}
	\tau_1^2(p_n,t_n)\ge \mu \kappa^2(p_n,t_n)\quad\text{ for all }n.
	\end{align*}
Then there exist constants $d_3,d_4,d_5>0$, depending only on $\mu$ and $\rho$, such that
\begin{enumerate}
\item for all $t\in\left[t_n,t_n+\frac{d_4}{M_{t_n}}\right]$,
	\begin{align*}
	|\tau_1(p_n,t)|\ge \frac{1}{\sqrt 2}|\tau_1(p_n,t_n)|,
	\end{align*}
\item for all $t\in\left[t_n,t_n+\frac{d_4}{M_{t_n}}\right]$ and all $p\in\gamma_t$ with $\operatorname{dist}_t(p,p_n)\le \sqrt{\frac{d_5}{M_{t_n}}},$
one has
	\begin{align*}
	|\tau_1(p,t)|\ge \frac{1}{\sqrt 2}|\tau_1(p_n,t)|,
	\end{align*}
\item hence for all $(p,t)\in \mathbf N(p_n,t_n,d_3)$,
	\begin{align*}
	|\tau_1(p,t)|\ge \frac{1}{2}|\tau_1(p_n,t_n)|.
	\end{align*}
\end{enumerate}
\end{theorem}

\begin{lemma}[Cut-off inequality]\label{lem_cutoff_tau}
Let $\eta=\eta(\cdot,t)\ge0$ be $C^2$ in $s$ and measurable in $t$, and assume that for each $t$ the support of $\eta(\cdot,t)$ is compactly contained in the interior of $\gamma_t$. Suppose moreover that $\eta_t=0$ in the chosen time-independent parametrisation. Then for every interval $[a,b]\subset[0,\omega)$,
	\begin{align*}
	\int_a^b\int_{\gamma_t}\eta \tau_1^2|\kappa|\,ds\,dt
\le
\int_{\gamma_a}\eta |\kappa|\,ds+ \int_a^b\int_{\gamma_t}|(\eta)_{ss}||\kappa|\,ds\,dt.
	\end{align*}
\end{lemma}

\begin{theorem}\label{thm_tau_over_kappa_to_zero}
Let $\{(p_n,t_n)\}$ be an essential blow-up sequence for a curve shortening flow with free boundary, as \eqref{essentialseq}. Then
	\begin{align*}
	\lim_{n\to\infty}\Bigl(\frac{\tau_1}{\kappa}\Bigr)(p_n,t_n)=0.
	\end{align*}
\end{theorem}

\begin{proof}
Assume for contradiction that there exists $\mu>0$ and a subsequence, still denoted $\{(p_n,t_n)\}$, such that
	\begin{align*}
	\tau_1^2(p_n,t_n)\ge \mu \kappa^2(p_n,t_n)\quad\text{ for all }n.
	\end{align*}

\medskip
\noindent
\emph{Step 1: an interior subsequence and a time shift.}
Discarding finitely many terms, we may assume $p_n\in \gamma_{t_n}\setminus\partial\gamma_{t_n}$ for all $n$ and
	\begin{align*}
	\operatorname{dist}_{t_n}(p_n,\partial\gamma_{t_n})>0.
	\end{align*}
Since $M_t=\max_{\gamma_t}\kappa^2(\cdot,t)\to\infty$ as $t\uparrow\omega$, we may choose $\tilde t_n<t_n$ so that
	\begin{align*}
	t_n=\tilde t_n+\frac{1}{16M_{\tilde t_n}}.
	\end{align*}
The purpose of $\tilde t_n$ is to accommodate the time delay in the scale-invariant estimates used below.

\medskip
\noindent
\emph{Step 2: persistence of curvature and torsion on a parabolic neighbourhood.}
By Theorem \ref{thm_disstorsion} together with the corresponding curvature persistence estimate (Theorem \ref{dissk}),
there exists $d=d(\mu,\rho)>0$ such that for all $n$,
	\begin{align*}
	|\kappa(p,t)|\ge \frac{1}{2}|\kappa(p_n,t_n)|
\quad\text{ and}\quad
|\tau_1(p,t)|\ge \frac{1}{2}|\tau_1(p_n,t_n)|
	\end{align*}
for all $(p,t)\in \mathbf N(p_n,t_n,d)$.

In particular, on $\mathbf N(p_n,t_n,d)$ we have
	\begin{align*}
	\tau_1^2|\kappa|
\ge
\frac{1}{8}|\tau_1(p_n,t_n)|^2|\kappa(p_n,t_n)|.
	\end{align*}

\medskip
\noindent
\emph{Step 3: distances do not collapse on time scales of order $M_{t_n}^{-1}$.}
Fix $u_1,u_2$ in a time-independent parametrisation and write $ds=v\,du$. Then $\partial_t v=-\kappa^2 v$ and hence
	\begin{align*}
	\frac{d}{dt}\operatorname{dist}_t(u_1,u_2)= -\int_{u_1}^{u_2}\kappa^2(\cdot,t)\,ds
\ge
- M_t\,\operatorname{dist}_t(u_1,u_2).
	\end{align*}
For $t\in\left[t_n,t_n+\frac{1}{4M_{t_n}}\right]$ we use $M_t\le 2M_{t_n}$ to obtain
	\begin{align*}
	\frac{d}{dt}\operatorname{dist}_t(u_1,u_2)\ge -2M_{t_n}\operatorname{dist}_t(u_1,u_2),
	\end{align*}
and therefore
	\begin{align*}
	\operatorname{dist}_t(u_1,u_2)\ge \operatorname{dist}_{t_n}(u_1,u_2)e^{-2M_{t_n}(t-t_n)}
\ge e^{-1/2}\operatorname{dist}_{t_n}(u_1,u_2).
	\end{align*}

\medskip
\noindent
\emph{Step 4: a uniform positive lower bound on a local space-time integral.}
Let $r_n=\sqrt{\frac{d}{M_{t_n}}}$. Using the distance control above, the spatial part of $\mathbf N(p_n,t_n,d)$
contains, for $t\in\left[t_n,t_n+\frac{d}{M_{t_n}}\right]$, an arc of length comparable to $r_n$
centred at $p_n$, with constants independent of $n$.
Consequently there exists a constant $C=C(\mu,\rho,d)>0$ such that
	\begin{align*}
	\iint_{\mathbf N(p_n,t_n,d)}\tau_1^2|\kappa|\,ds\,dt \ge C
\quad\text{ for all sufficiently large }n.
	\end{align*}

\medskip
\noindent
\emph{Step 5: a cut-off argument to avoid boundary terms and force smallness near extinction.}
Choose $\eta_n=\eta_n(\cdot)$ at time $t_n$ with the following properties:
	\begin{align*}
	\eta_n\equiv 1\ \text{ on}\ \{\operatorname{dist}_{t_n}(p_n,\cdot)\le r_n\},\qquad
\eta_n\equiv 0\ \text{ on}\ \{\operatorname{dist}_{t_n}(p_n,\cdot)\ge 2r_n\},
	\end{align*}
and in addition
	\begin{align*}
	\eta_n\equiv 0\ \text{ in a neighbourhood of}\ \partial\gamma_{t_n}.
	\end{align*}
Transport $\eta_n$ in the fixed parametrisation so that $(\eta_n)_t=0$.
Then $\eta_n(\cdot,t)$ is supported in the interior of $\gamma_t$ for $t$ close to $t_n$,
so Lemma \ref{lem_cutoff_tau} applies on the interval
	\begin{align*}
	I_n=\left[t_n,t_n+\frac{d}{M_{t_n}}\right].
	\end{align*}
Moreover, we may arrange
	\begin{align*}
	|(\eta_n)_{ss}|\le \frac{c}{r_n^2}
	\end{align*}
with a universal constant $c$.

Applying Lemma \ref{lem_cutoff_tau} and using $M_t\le 2M_{t_n}$ on $I_n$, we get
	\begin{align*}
	\int_{I_n}\int_{\gamma_t}\eta_n \tau_1^2|\kappa|\,ds\,dt
\le
\int_{\gamma_{t_n}}\eta_n |\kappa|\,ds+ \int_{I_n}\int_{\gamma_t}|(\eta_n)_{ss}||\kappa|\,ds\,dt.
	\end{align*}
Since $\operatorname{supp}\eta_n$ has length at most $4r_n$ and $|\kappa|\le \sqrt{2M_{t_n}}$ on $I_n$, we have
	\begin{align*}
	\int_{\gamma_{t_n}}\eta_n |\kappa|\,ds \le 4r_n\sqrt{2M_{t_n}}=4\sqrt 2\sqrt d,
	\end{align*}
and similarly
	\begin{align*}
	\int_{I_n}\int_{\gamma_t}|(\eta_n)_{ss}||\kappa|\,ds\,dt
\le
\frac{c}{r_n^2}\cdot \frac{d}{M_{t_n}}\cdot 4\sqrt 2\sqrt d= 4c\sqrt 2\sqrt d.
	\end{align*}
Hence
	\begin{align*}
	\int_{I_n}\int_{\gamma_t}\eta_n \tau_1^2|\kappa|\,ds\,dt \le c_0\sqrt d,
	\end{align*}
where $c_0=4\sqrt 2(1+c)$ is independent of $n$.

\medskip
\noindent
\emph{Step 6: contradiction.}
Because $\eta_n\equiv 1$ on $\mathbf N(p_n,t_n,d)$, we have
	\begin{align*}
	\iint_{\mathbf N(p_n,t_n,d)}\tau_1^2|\kappa|\,ds\,dt
\le
\int_{I_n}\int_{\gamma_t}\eta_n \tau_1^2|\kappa|\,ds\,dt
\le c_0\sqrt d.
	\end{align*}
Choosing $d>0$ small so that $c_0\sqrt d<C$, this contradicts Step 4. Therefore no such $\mu>0$ can exist, and thus
	\begin{align*}
	\Bigl(\frac{\tau_1}{\kappa}\Bigr)(p_n,t_n)\to 0
	\end{align*}
along every essential blow-up sequence in the interior.

We now treat the boundary case. In particular, along any boundary blow-up sequence, the torsion is negligible compared to the curvature. By Lemma \ref{lem_endpoint-control}, we have a uniform bound on the torsion on the boundary:
	\begin{align*}
|\tau_1(p, t)| \leq K \quad \text{ for all }p \in \partial \gamma_t \text{ and for all } t\in [0,\omega)
	\end{align*}
We fix the boundary blow-up sequence $\left(p_j, t_j\right) \in \partial \gamma_{t_j}$ with $\left|\kappa\left(p_j, t_j\right)\right| \rightarrow \infty$ which ensures that that $\tau_1$ is well defined along this sequence. At each such point we estimate
	\begin{align*}
\left|\frac{\tau_1}{\kappa}\right|\left(p_j, t_j\right)=\frac{\left|\tau_1\left(p_j, t_j\right)\right|}{\left|\kappa\left(p_j, t_j\right)\right|} \leq \frac{K}{\left|\kappa\left(p_j, t_j\right)\right|}.
	\end{align*}
Since $\left|\kappa\left(p_j, t_j\right)\right| \rightarrow \infty$, the right-hand side tends to 0 , hence
	\begin{align*}
\frac{\tau_1}{\kappa}\left(p_j, t_j\right) \to 0,
	\end{align*}
as claimed.
\end{proof}

Therefore, from \cite{Altschuler91}, we know that a limit $\gamma_\infty$ exists and that is a family of planar, convex curves.

\subsection{Hamilton argument}\label{hamiltonsection}
In this section we apply Richard S. Hamilton’s Harnack inequality to the planar convex interior eternal limit obtained after (i) boundary Type II blow-up, (ii) passing to a free-boundary limit in a half-plane, and (iii) reflecting across the boundary line.

\begin{theorem}[Hamilton \cite{Hamilton}]\label{thm_Hamilton-Harnack}
Let $\{M_t\}_{t>0}$ be a smooth mean curvature flow of hypersurfaces in $\mathbb R^{n+1}$ with
nonnegative second fundamental form. Then for every tangent vector field $V$ on $M_t$ one has
	\begin{align*}
	\partial_t H+\frac{1}{2t}H+2\langle \nabla H,V\rangle + h(V,V)\ge 0,
	\end{align*}
where $H$ is the mean curvature, $\nabla$ is the Levi--Civita connection on $M_t$, and $h$ is the
second fundamental form.
\end{theorem}

\begin{remark}[Specialisation to curve shortening flow]
For a curve $\gamma_t\subset\mathbb R^{n+1}$ evolving by curve shortening flow, the mean curvature is
the scalar curvature $\kappa$ and the second fundamental form is $h=\kappa\, g$, where $g$ is the
induced metric on the curve. Thus $h(V,V)=\kappa|V|^2$. In this one-dimensional setting the Harnack
inequality takes the form
	\begin{align*}
	\partial_t\kappa+\frac{1}{2t}\kappa+2\langle \nabla \kappa,V\rangle+\kappa|V|^2\ge 0,
	\end{align*}
for all tangent vector fields $V$ along $\gamma_t$.
If we time-shift the flow to start at $t=t_0$, i.e. consider $\tilde\gamma_t=\gamma_{t_0+t}$
for $t>0$, then the same inequality reads
	\begin{align*}
	\partial_t\kappa+\frac{1}{2(t-t_0)}\kappa+2\langle \nabla \kappa,V\rangle+\kappa|V|^2\ge 0
\qquad (t>t_0).
	\end{align*}
\end{remark}
Following Hamilton, we set $Z:=\partial_t \kappa+2\langle \nabla \kappa, V\rangle+\kappa|V|^2$ and $\widetilde{Z}=Z+\frac{1}{2 (t-t_0)} \kappa$. Note that $Z\ge-\frac{1}{2(t-t_0)}\kappa$. If the solution is weakly convex and eternal, then since it exists for time $t >t_0$, when $t_0=-\infty$ the Harnack estimate implies that $Z > 0$ everywhere for all $V$. If equality occurs at some spacetime point for some $V$, then by Hamilton’s strong maximum principle for the Harnack quadratic (together with bounded curvature) the flow is a translator. We now assume our solution is eternal and strictly convex, so that $\kappa>0$.
	\begin{lemma}[\cite{Hamilton}]
If $F \geq 0$ is a weakly positive function on $M$ satisfying $\left(D_t-\Delta\right) F=0$ and if $Z \geq F$ at time $\alpha$ for all $V$, then $Z \geq F$ at all subsequent times for all $V$.
	\end{lemma}
If the curvature $\kappa$ assumes its maximum at a point in space-time, then at that point, we have $\partial_t \kappa=0$ and $\nabla\kappa=0$, so $Z=0$ in the direction $V=0$. Thus the strong maximum principle implies that there must be exactly one $V$ at each point in space-time where $Z=0$, and $V$ will vary smoothly. Also, since $Z(V)=0$ at a maximum point for some $V$, then $\frac{\partial Z}{\partial V}=0$.

Since $\kappa>0$, the map $V \mapsto Z(V)$ is a strictly convex quadratic, hence admits a unique minimizer $V^*$. If $Z \geq 0$ and $Z(V)=0$ at some point for some $V$, then $V=V^*$ and $\partial_V Z=0$, yielding \eqref{eq1} and the minimal value satisfies \eqref{eq2}.

Differentiating $Z$ with respect to $V$, we get the following:
	\begin{align}\label{eq1}
	\frac{\partial Z}{\partial V}=2\nabla\kappa+2\kappa V=0 \Longleftrightarrow V^*=-\frac{\nabla\kappa}{\kappa}.
	\end{align}
Putting $V^*$ back into $Z$, we have
	\begin{align}\label{eq2}
	Z(V^*)&=\partial_t \kappa-\frac{|\nabla\kappa|^2}{\kappa}=0.
	\end{align}
\begin{theorem}[\cite{Hamilton}]
If we have a solution to the curve shortening flow which is strictly convex, and a vector field $V$ satisfying
	\begin{align*}
\nabla\kappa+\kappa V=0
	\end{align*}
and
	\begin{align*}
\partial_t \kappa-\frac{|\nabla\kappa|^2}{\kappa}=0,
	\end{align*}
then the solution is a translating soliton.
\end{theorem}
In the one-dimensional setting, Hamilton's rigidity conditions reduce precisely to $\nabla \kappa+\kappa V=0$ and the corresponding equality-case propagation for the Harnack quadratic. Bounded curvature on $(-\infty, \infty)$ ensures the Harnack quadratic attains its infimum and allows the strong maximum principle/equality-case argument to be applied on the eternal limit. Due to Lemma \ref{bddcurv}, we have that the eternal solution is a translating soliton.
	\begin{lemma}[\cite{Altschuler91}]\label{bddcurv}
For a type-II singularity, there exists an essential blow-up sequence along which the rescaled solutions converge to a solution whose curvature is bounded on $[-\infty,+\infty]$.
	\end{lemma}
\subsection{Concluding argument}
Assume throughout that the reflected Gaussian entropy satisfies
	\begin{align}\label{entropybound}
\operatorname{Ent}_{\partial \Omega}\left[\left\{\gamma_t\right\}\right]<2.
	\end{align}
By the reflected monotonicity formula, the reflected entropy is nonincreasing along the flow, and every (boundary) tangent flow arising from any blow-up has reflected entropy bounded above by \eqref{entropybound}.

We first rule out Type II singularities under this entropy bound. Indeed, if a boundary Type II singularity occurred at time $T<\infty $, then by the Type II blow-up analysis in subsection \ref{hamiltonsection} rescaling at the curvature scale and passing to a limit produces a smooth eternal free-boundary limit in a half-plane whose reflection is an eternal interior solution. Hamilton's Harnack inequality forces this reflected limit to be a translator and Theorem \ref{thm_tau_over_kappa_to_zero} shows that it must in fact be the Grim Reaper translator. In particular, the corresponding tangent flow has Gaussian entropy at least $ 2 $. Since the reflected entropy of the original free-boundary flow controls the entropy of such reflected tangent flows by Lemma \ref{lem_entropy_to_tangent}, this yields
	\begin{align*}
\operatorname{Ent}_{\partial \Omega}\left[\left\{\gamma_t\right\}\right] \geq 2
	\end{align*}
contradicting \eqref{entropybound}. Therefore, no Type II singularity can occur.

Consequently, if the maximal time satisfies $T<\infty $, then any singularity must be of Type I. Under the entropy bound \eqref{entropybound}, the Type I boundary blow-up analysis shows that the (reflected) tangent shrinker is the round circle, hence the boundary tangent flow is the unit semicircle in the tangent half-space. In particular the diameter of $ \gamma_t $ tends to $0 $ and $ \gamma_t $ collapses to a single boundary point $z\in\partial\Omega$, and the rescaled curves
	\begin{align*}
\widetilde{\gamma}_t=\frac{\gamma_t-z}{\sqrt{2(T-t)}}
	\end{align*}
converge smoothly as $t\to T$ to the unit semicircle in $T_z\Omega\simeq\mathbb R^{n+1}_+$, proving alternative (b).

Finally, if no finite-time singularity occurs, then $T=\infty$, and the long-time convergence statement in alternative (a) follows from the chord-convergence theorem of Langford-Zhu for free-boundary curve shortening flow in strictly convex domains \cite{LanZhu}. This completes the proof of Theorem \ref{thm_FBCSF-classification}.

\end{document}